\documentclass[12pt]{amsart}
\usepackage{moreverb,url}
\usepackage{amssymb,amsmath,amsfonts}
\usepackage{bm}
\usepackage[font=scriptsize, labelfont=bf]{caption}
\usepackage{float}
\usepackage{afterpage}
\usepackage{mathrsfs}
\usepackage{subcaption}
\usepackage[squaren]{SIunits}
\usepackage{verbatim}
\usepackage[square,numbers]{natbib}
\usepackage{stmaryrd}
\usepackage{fancyhdr}
\usepackage{syntax,etoolbox}
\usepackage{graphicx}
\usepackage{grffile}
\usepackage{hyperref}
\hypersetup{colorlinks=true,allcolors=blue}
\usepackage[square,numbers]{natbib}
\usepackage[left=1in, right=1in, bottom=1in,top=1in]{geometry}
\usepackage{lineno}
\usepackage{syntax,etoolbox}
\usepackage{algorithmicx}
\usepackage{algorithm}
\usepackage{algpseudocode}

\newtheorem{theorem}{Theorem}[section]
\newtheorem{lemma}{Lemma}[section]

\newtheorem{corollary}[theorem]{Corollary}

\newtheorem{innercustomgeneric}{\customgenericname}
\providecommand{\customgenericname}{}
\newcommand{\newcustomproblem}[2]{%
	\newenvironment{#1}[1]
	{%
		\renewcommand\customgenericname{#2}%
		\renewcommand\theinnercustomgeneric{##1}%
		\innercustomgeneric
	}
	{\endinnercustomgeneric}
}

\newcustomproblem{customprob}{Problem}

\newcommand*{\bqed}{\hfill\ensuremath{\blacksquare}}%
\newcommand{\wsc}{\overset{\ast}{\rightharpoonup}}%

\def\dd{\, \mathrm{d}}

\algblock{Input}{EndInput}

\algnotext{EndInput}

\algblock{Output}{EndOutput}

\algnotext{EndOutput}

\algblock{Iterate}{EndIterate}

\algnotext{EndIterate}

\newcommand{\Desc}[2]{\State \makebox[2em][l]{#1}#2}

\allowdisplaybreaks

\begin{document}
	
	
	\title[Icosahedral viral capsids modelling]{A three-dimensional discrete model for approximating the deformation of a viral capsid subjected to lying over a flat surface in the static and time-dependent case}
	

	\author[Piersanti]{Paolo Piersanti}
	\address{Department of Mathematics and Institute for Scientific Computing and Applied Mathematics, Indiana University Bloomington, 729 East Third Street, Bloomington, Indiana, USA}
	\email{ppiersan@iu.edu}

	\author[White]{Kristen White}
	\address{Department of Chemistry, Indiana University Bloomington, 800 East Kirkwood Avenue, Bloomington, Indiana 47405, USA}
	\email{kw98@iu.edu}
	
	\author[Dragnea]{Bogdan Dragnea}
	\address{Department of Chemistry, Indiana University Bloomington, 800 East Kirkwood Avenue, Bloomington, Indiana 47405, USA}
	\email{dragnea@indiana.edu}
		
	\author[Temam]{Roger Temam}
	\address{Department of Mathematics and Institute for Scientific Computing and Applied Mathematics, Indiana University Bloomington, 729 East Third Street, Bloomington, Indiana, USA}
	\email[Corresponding author]{temam@indiana.edu}

\begin{abstract}
In this paper we present a three-dimensional discrete model governing the deformation of a viral capsid, modelled as a regular icosahedron and subjected not to cross a given flat rigid surface on which it initially lies in correspondence of one vertex only. First, we set up the model in the form of a set of variational inequalities posed over a non-empty, closed and convex subset of a suitable space. Secondly, we show the existence and uniqueness of the solution for the proposed model. Thirdly, we numerically test this model and we observe that the outputs of the numerical experiments comply with physics. Finally, we establish the existence of solutions for the corresponding time-dependent obstacle problem.
\end{abstract}

\maketitle

\section{Introduction}

Canonical virus architecture often involves a symmetric polyhedral cage surrounding the viral genome. This nanoscopic cage, called the capsid, takes multiple roles during the virus life cycle: genome protection, targeting of a host cell, genome presentation. Switching between these roles is triggered by reading the chemical environment and usually manifests as a change in the mechanical properties of the cage \cite{Zandi2020}. The latter can be probed by atomic force microscopy (AFM)  -- an \emph{in situ} method able to measure the stress-strain relationship at the scale of a single virus, as a function of the chemical environment. In AFM imaging, a sharp mechanical probe compresses the capsid uniaxially against a flat solid support. As the virus deforms, the contact area and the magnitude of adhesive forces increase. The final shape of the cage is the result of the balance between adhesive and elastic forces \cite{Zeng2017a}. Here we present a purely elastic icosahedral cage deformation model including irreversible adhesion, for studying the shape of polyhedral cages that have undergone a two step process which mimics force application in AFM imaging. 

This paper is divided into five sections, including this one. In section~\ref{Sec:1} we present the main notation ans well as the geometrical and analytical background.
In section~\ref{Sec:2} we derive a discrete linearized static (i.e., time-independent) model governing the deformation of an icosahedral viral cage under the action of forces like those described in the above paragraph, and subjected to a \emph{confinement condition} according to which the points of the deformed reference configuration do not have to cross a prescribed plane on which the reference configuration of the icosahedral cage lies in correspondence of one point only when no forces are acting on it. This problem can be classified as an \emph{obstacle problem}, and it is thus possible to observe that it takes the form of a set of variational inequalities posed over a non-empty, closed and convex subset of an \emph{ad hoc} Euclidean (finite-dimensional) space.
We then establish the existence and uniqueness of solutions for this model. In section~\ref{Sec:3} we conduct numerical experiments on the model we recovered in section~\ref{Sec:2}. In the first series of experiments we compute the deformed reference configuration of the icosahedral viral cage when it is subjected to the action of applied body forces and -- at the same time -- it has to obey the confinement condition introduced beforehand. In the second series of experiments, we compute the equilibrium shape of the viral cage in the case of \emph{irreversible adhesion}, in the sense that the points which are in contact with the obstacle (the prescribed plane) at the end of the first series of experiments, must continue to remain in contact with the obstacle when the second series of experiments is carried out.
Finally, in section~\ref{Sec:4}, we study the time-dependent version of the model introduced in section~\ref{Sec:2}. Unlike the elliptic counterpart, the concept of solution for this time-dependent problem -- which will be of hyperbolic type since the displacement is the main unknown entering the model under consideration -- is \emph{a priori} not clear. Therefore, we will recover the governing hyperbolic contact problem and we will establish the existence of solutions using a technique based on the penalty method. It is noticeable that the concept of solution for this newly recovered model is not standard, as it involves the usage of vector-valued measures.

\section{Geometrical and analytical preliminaries}
\label{Sec:1}

A point $A$ in the plane corresponds to a column vector in $\mathbb{R}^3$ of the form $A=\begin{pmatrix}x_A,y_A,z_A\end{pmatrix}^T$. Here, the symbol $T$ denotes the transposition operator.
Let us consider a Cartesian frame for the three-dimensional plane with origin $O=\begin{pmatrix}0,0,0\end{pmatrix}^T$ and with canonical directions $\vec{e}_1=\begin{pmatrix}1,0,0\end{pmatrix}^T$, $\vec{e}_2=\begin{pmatrix}0,1,0\end{pmatrix}^T$ and $\vec{e}_3=\begin{pmatrix}0,0,1\end{pmatrix}^T$.

The position vector associated with the point $A$ is denoted by $\overrightarrow{OA}$; the angle between three points $A$, $B$ and $C$ with vertex at $B$ is either denoted by $\widehat{ABC}$ or by a Greek letter.
The Euclidean inner product and the vector product between two vectors $\overrightarrow{OA}$ and $\overrightarrow{OB}$ are respectively denoted by $\overrightarrow{OA} \cdot \overrightarrow{OB}=\overrightarrow{OA}^T \overrightarrow{OB}$ and $\overrightarrow{OA} \times \overrightarrow{OB}$. The Euclidean norm of $\overrightarrow{OA}$ is denoted $\left|\overrightarrow{OA}\right|$. Matrices are denoted by capital Greek letters.
Tensors are denoted by boldface capital Latin letters.

Given an open interval $I$ of $\mathbb{R}$, notations such as $L^p(I)$, $W^{m,p}(I)$, $m,p \ge 1$, designate the usual Lebesgue and Sobolev spaces, with norms $\|\cdot\|_{L^p(0,T)}$ and $\|\cdot\|_{W^{m,p}(0,T)}$, respectively.
The space of continuous functions on $\overline{I}$ is denoted by $\mathcal{C}^0(\overline{I})$.
The spaces $L^p(I;\mathbb{R}^n)$ and $W^{m,p}(I;\mathbb{R}^n)$ are spaces of vector-valued functions $\bm{v}:I\to\mathbb{R}^n$ such that each component $v_i$, $1 \le i \le n$, is in $L^p(I)$ or $W^{m,p}I(I)$, respectively.
The space $\mathcal{C}^0(\overline{I};\mathbb{R}^n)$ is the space of vector-valued functions $\bm{v}:I\to\mathbb{R}^n$ such that each component $v_i$, $1 \le i \le n$, is in $\mathcal{C}^0(\overline{I})$. The space $\mathcal{D}(I)$ denotes the space of infinitely differentiable functions with compact support in $I$.

The positive and negative parts of a function $f:I \to \mathbb{R}$ are respectively denoted by:
$$
f^{+}(x):=\max\{f(x),0\}\quad\textup{ and }\quad f^{-}(x):=-\min\{f(x),0\} \quad x \in I.
$$

In this paper we model the deformation of a regular icosahedron whose vertices (and so the edges) are subjected not to cross an \emph{undeformable} flat surface. The problem amounts to minimizing an \emph{ad hoc} energy functional defined over a non-empty, closed, convex subset of a finite-dimensional space or, equivalently, to solving a set of variational inequalities posed on the aforementioned non-empty, closed, convex set.

We assume that one and only one vertex of the undeformed reference configuration is initially in contact with the flat surface. 
We also assume that each edge is \emph{massless} and can \emph{only stretch or compress}, hence, that \emph{there is no torsion acting on them}.
We further \emph{assume} that one such contact point, denoted by $P_0$ in what follows, undergoes no displacement; this assumption is critical to establish the existence and uniqueness of the solution of the governing equations.

\begin{figure}
	\includegraphics[width=0.45\linewidth]{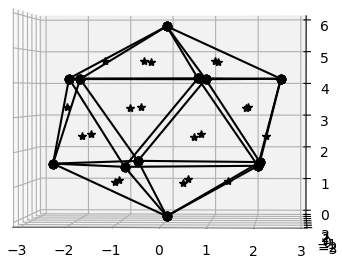}
	\caption{The reference configuration of the problem under consideration is a regular icosahedron. The points shaped like asterisks denote the barycentres of each triangular face of the icosahedron.}
	\label{fig:0}
\end{figure}

\section{Formulation and well-posedness of the corresponding three-dimensional discrete model}
\label{Sec:2}

When a vertex $P$ of the regular icosahedron under consideration undergoes the action of an applied body force, it is mapped onto a new point in the space, denoted by $P'$. 

Let $\bm{F}=(\vec{f}_i)_{i=1}^{11} \in \mathbb{R}^{33}$ denote the array of applied body forces acting on the regular icosahedron vertices. The application of the force vector $\vec{f}_i \in \mathbb{R}^3$ on the point $P_i$ displaces the position vector $\overrightarrow{OP_i}$ by a vector $\vec{u}_i$, and transforms the  vector $\overrightarrow{OP_i}$ into the  vector $\overrightarrow{OP_i'}$ via the following relation:
\begin{equation*}
	\overrightarrow{OP_i'} = \overrightarrow{OP_i} + \vec{u}_i,\quad\textup{ for each } 1 \le i \le 11.
\end{equation*}

We denote by $\ell$ the length of any edge of the undeformed reference configuration of the regular icosahedron.

Since the point $P_0$ undergoes, by assumption, no deformation we let $\vec{u}_0=(0,0,0)^T$.
For each $1 \le i \le 11$, define the set
$$
\mathscr{N}(i):=\{j \neq i; |\overrightarrow{P_iP_j}|=\ell\}.
$$

The total elastic energy of the mechanical system under consideration is contributed by three components: the stretching (or compression) of each edge, the variation of the amplitude of each dihedral angle and, finally, the variation of the Gaussian curvature at each vertex.
We recall that the Gaussian curvature of a polyhedron is computed by means of a formula originally discovered by Descartes, and known in the literature as \emph{the Descartes Lost Theorem}, also known as \emph{The Second Euler Theorem}, that was later generalized by Gauss and Bonnet to a more general context (cf., e.g., Theorem 6.1.7 of~\cite{AbateTovena2012}). 
Let us recall that, for a general polyhedron, the \emph{angular defect at a vertex} is defined as the difference between $2\pi$ and the sum of all the face-angles at the vertex. 
We recall that, if a polyhedron is convex, then the angular defect at each vertex is always positive.

We now recall the celebrated formula proved by Descartes (cf., e.g., page~145 of~\cite{Prasolov2001}).

\begin{theorem}[The Lost Descartes Theorem]
	\label{Descartes}
	Let $P$ be the vertex of a convex three-dimensional polyhedron. Let $D(P)$ denote the angular defect at the vertex $P$.
	Then, the curvature at the vertex $P$ is given by:
	$$
	K(P)=D(P).
	$$
	
	Moreover, if the polyhedron is convex, its Gaussian curvature, which is given by the sum of the curvatures at each of the polyhedron vertex, is always equal to $4\pi$.
	\qed
\end{theorem}

The first component of the total elastic energy is the stretching energy associated with the displacement 
$$
\bm{U}=\begin{pmatrix}
	\vec{u}_1\\
	\vdots\\
	\vec{u}_{11}
\end{pmatrix} \in \mathbb{R}^{33} \textup{ with }
\vec{u}_i=
\begin{pmatrix}
u_i^1\\u_i^2\\u_i^3
\end{pmatrix} \textup{ for all }1 \le i \le 11.
$$ 

The stretching energy is computed via Hooke's law, i.e.,
\begin{equation*}
	J_s(\bm{U}):=\dfrac{k_s}{4} \sum_{i=0}^{11} \sum_{j \in \mathscr{N}(i)} |\vec{u}_i - \vec{u}_j|^2,
\end{equation*}
where the elastic constant $k_s>0$ is associated with the elongation properties of the constitutive material, and the nature of the energy is aptly recalled by the subscript ``$s$''. The convexity of $J_s$ is straightforward: Indeed, for all $0 \le \lambda \le 1$ and all $\bm{U}$ and $\bm{V}$ in $\mathbb{R}^{33}$, we have that
\begin{equation*}
\begin{aligned}
&J_s(\lambda \bm{U} + (1-\lambda) \bm{V})=\dfrac{k_s}{4} \sum_{i=0}^{11} \sum_{j \in \mathscr{N}(i)} |\lambda(\vec{u}_i - \vec{u}_j) + (1-\lambda)(\vec{v}_i - \vec{v}_j)|^2\\
&\le \lambda\dfrac{k_s}{4} \sum_{i=0}^{11} \sum_{j \in \mathscr{N}(i)} |\vec{u}_i - \vec{u}_j|^2 + (1-\lambda)\dfrac{k_s}{4} \sum_{i=0}^{11} \sum_{j \in \mathscr{N}(i)} |\vec{v}_i - \vec{v}_j|^2\\
&=\lambda J_s(\bm{U}) +(1-\lambda) J_s(\bm{V}).
\end{aligned}
\end{equation*}

Observe that the functional $J_s$ is differentiable in the Fr\'echet sense at any $\bm{U}$, since
\begin{equation*}
\begin{aligned}
&J_s(\bm{U}+\bm{V})-J_s(\bm{U})=\dfrac{k_s}{4}\sum_{i=0}^{11} \sum_{j \in \mathscr{N}(i)} \left\{\left|(\vec{u}_i-\vec{u}_j)+(\vec{v}_i-\vec{v}_j)\right|^2-|\vec{u}_i-\vec{u}_j|^2\right\}\\
&=\dfrac{k_s}{2}\sum_{i=0}^{11} \sum_{j \in \mathscr{N}(i)} (\vec{u}_i-\vec{u}_j) \cdot (\vec{v}_i-\vec{v}_j)
+\dfrac{k_s}{4}\sum_{i=0}^{11} \sum_{j \in \mathscr{N}(i)} |\vec{v}_i-\vec{v}_j|^2,
\end{aligned}
\end{equation*}
and, therefore, we have that the action of the Fr\'echet derivative of $J_s'(\bm{U})$ at any $\bm{V}$ is given by:
\begin{equation*}
J_s'(\bm{U})\bm{V}=\dfrac{k_s}{2}\sum_{i=0}^{11} \sum_{j \in \mathscr{N}(i)} (\vec{u}_i-\vec{u}_j) \cdot (\vec{v}_i-\vec{v}_j).
\end{equation*}

To prove the strict convexity, let us show that for all  $\bm{U}$ and $\bm{V}$ with $\bm{U} \neq \bm{V}$ it results:
\begin{equation}
\label{euler}
J_s(\bm{V}) > J_s(\bm{U})+J_s'(\bm{U})(\bm{V}-\bm{U}).
\end{equation}

Observe that a direct computation gives:
\begin{equation}
\label{euler1}
\begin{aligned}
&J_s(\bm{V}) - J_s(\bm{U})-J_s'(\bm{U})(\bm{V}-\bm{U})\\
&=\dfrac{k_s}{4} \sum_{i=0}^{11} \sum_{j \in \mathscr{N}(i)}\left\{|\vec{v}_i-\vec{v}_j|^2-|\vec{u}_i-\vec{u}_j|^2-2(\vec{u}_i-\vec{u}_j)\cdot(\vec{v}_i-\vec{v}_j)+2|\vec{u}_i-\vec{u}_j|^2\right\}\\
&=\dfrac{k_s}{4} \sum_{i=0}^{11} \sum_{j \in \mathscr{N}(i)}\left\{|\vec{v}_i-\vec{v}_j|^2+|\vec{u}_i-\vec{u}_j|^2-2(\vec{u}_i-\vec{u}_j)\cdot(\vec{v}_i-\vec{v}_j)\right\}\\
&=\dfrac{k_s}{4} \sum_{i=0}^{11} \sum_{j \in \mathscr{N}(i)}\left|(\vec{u}_i-\vec{u}_j)-(\vec{v}_i-\vec{v}_j)\right|^2 \ge 0.
\end{aligned}
\end{equation}

If for all $1 \le i \le 11$ and all $j \in \mathscr{N}(i)$ it resulted $(\vec{u}_i-\vec{u}_j)=(\vec{v}_i-\vec{v}_j)$ then, in particular, for all $1 \le i \le 11$ such that $0 \in \mathscr{N}(i)$ the assumption $\vec{u_0}=\vec{v}_0=\vec{0}$ would imply
\begin{equation}
\label{cond2}
\vec{u}_i=\vec{v}_i.
\end{equation}

Thanks to~\eqref{cond2}, we in turn derive that $\vec{u}_j=\vec{v}_j$ for all $j \in \mathscr{N}(i) \setminus \{0\}$. Therefore, given any $k \in \mathscr{N}(i) \setminus\{0\}$, we have that
\begin{equation}
\label{cond3}
\vec{u}_\ell =\vec{v}_\ell,\quad\textup{ for all }\ell \in \mathscr{N}(k).
\end{equation}

By repeatedly applying~\eqref{cond2} and~\eqref{cond3} we obtain that the left-hand side of~\eqref{euler1} is equal to zero if and only if $\bm{U}=\bm{V}$ (the ``if'' part is straightforward). Therefore, if $\bm{U} \neq \bm{V}$, the inequality~\eqref{euler} is verified and we have that $J_s$ is strictly convex, as it was to be proved.

%
%

The second component of the total elastic energy is associated with the variation of the dihedral angle $\alpha$ between any pair of adjacent faces.
From now on, we will refer to this kind of energy as \emph{the bending energy}.
The corresponding bending energy is given by:
\begin{equation*}
	\dfrac{k_b}{2} |\alpha-\alpha'|^2, \quad 1 \le m \le 20,
\end{equation*}
where the symbol $\alpha$ denotes the measure of any dihedral angle of the reference configuration, the symbol $\alpha'$ denotes the measure of the angle into which $\alpha$ is transformed after the application of an applied body force, and the symbol $k_b>0$ denotes the elastic constant associated with the bending properties of the constitutive material.

If the difference between $\alpha$ and $\alpha'$ is small, we can approximate $\alpha-\alpha'$ by $\sin(\alpha-\alpha')$. The latter term has the advantage that it can be expressed in terms of a vector product. In order to formulate the bending energy variation we resort to the outer unit normal vectors associated with each face of the reference configuration. More specifically, let the points $G^m$, $1 \le m \le 20$, denote the barycentres of the icosahedron faces.
Let $g$ denote the distance between the barycentres of two adjacent faces. Let $G^{m_1}$ and $G^{m_2}$ be the barycentres of any pair of adjacent faces. 

The integers $i=i(m_1,m_2)$ and $j=j(m_1,m_2)$ range between 0 and 11 and are associated with the reference configuration vertices in the following fashion:
$$
|\overrightarrow{G^{m_1}P_i}|=|\overrightarrow{G^{m_1}P_j}|=|\overrightarrow{G^{m_2}P_i}|=|\overrightarrow{G^{m_2}P_j}|.
$$

It is easy to see that the measure of the dihedral angle between two adjacent faces is equal to $\pi$ minus the measure of the dihedral angle between the outer unit normal vectors associated with two such faces.
The unit normal vector associated with the face of the reference configuration whose barycentre is the point $G^m$ is given by:
$$
\vec{q}^{\,\,m}:=\dfrac{\overrightarrow{G^m P_i} \times \overrightarrow{G^m P_j}}{|\overrightarrow{G^m P_i} \times \overrightarrow{G^m P_j}|}.
$$

When the reference configuration given by the regular icosahedron under consideration undergoes a deformation, the barycentre of any face of the reference configuration is deformed onto the barycentre of the image of that very face (which is again a triangle) as a result of the applied deformation. This makes sense from the geometrical and physical point of view, as the barycentre is not a material point on which the applied body  forces act.
For each $1 \le m \le 20$, define the set
$$
\mathscr{M}(m):=\{0 \le r \le 11; |\overrightarrow{G^mP_r}| \textup{ is minimal}\}.
$$

Any barycentre $G^m$ is thus transformed onto a point $G^{'m}$ via the following transformation:
$$
\overrightarrow{OG^{'m}}=\overrightarrow{OG^{m}}+\sum_{r \in \mathscr{M}(m)} \dfrac{\vec{u}_r}{3}.
$$

The unit normal vector associated with the face of the deformed reference configuration whose barycentre is the point $G^{'m}$ is given by:
$$
\vec{q}^{\,\,'m}:=\dfrac{\overrightarrow{G^{'m} P_i'} \times \overrightarrow{G^{'m} P_j'}}{|\overrightarrow{G^{'m} P_i'} \times \overrightarrow{G^{'m} P_j'}|}.
$$

The variation of the dihedral angle $\alpha$ between any pair of adjacent faces having for barycentres the points $G^{m_1}$ and $G^{m_2}$, with $1 \le m_1, m_2 \le 20$ is thus given by:

\begin{equation}
	\label{angle:1}
	\begin{aligned}
		&\alpha-\alpha' \approx \sin(\alpha -\alpha')=\sin \alpha \cos \alpha'-\cos\alpha \sin\alpha'\\
		&=|\vec{q}^{\,\,'m_1} \times \vec{q}^{\,\,'m_2}|(\vec{q}^{\,\,m_1} \cdot \vec{q}^{\,\,m_2})-(\vec{q}^{\,\,'m_1} \cdot \vec{q}^{\,\,'m_2}) |\vec{q}^{\,\,m_1} \times \vec{q}^{\,\,m_2}|\\
		&=\left|\dfrac{\overrightarrow{G^{'m_1} P_i'} \times \overrightarrow{G^{'m_1} P_j'}}{|\overrightarrow{G^{'m_1} P_i'} \times \overrightarrow{G^{'m_1} P_j'}|} \times \dfrac{\overrightarrow{G^{'m_2} P_j'} \times \overrightarrow{G^{'m_2} P_i'}}{|\overrightarrow{G^{'m_2} P_j'} \times \overrightarrow{G^{'m_2} P_i'}|}\right|
		\left(\dfrac{\overrightarrow{G^{m_1} P_i} \times \overrightarrow{G^{m_1} P_j}}{|\overrightarrow{G^{m_1} P_i} \times \overrightarrow{G^{m_1} P_j}|} \cdot \dfrac{\overrightarrow{G^{m_2} P_j} \times \overrightarrow{G^{m_2} P_i}}{|\overrightarrow{G^{m_2} P_j} \times \overrightarrow{G^{m_2} P_i}|}\right)\\
		&\quad -\left|\dfrac{\overrightarrow{G^{m_1} P_i} \times \overrightarrow{G^{m_1} P_j}}{|\overrightarrow{G^{m_1} P_i} \times \overrightarrow{G^{m_1} P_j}|} \times \dfrac{\overrightarrow{G^{m_2} P_j} \times \overrightarrow{G^{m_2} P_i}}{|\overrightarrow{G^{m_2} P_j} \times \overrightarrow{G^{m_2} P_i}|}\right|
		\left(\dfrac{\overrightarrow{G^{'m_1} P_i'} \times \overrightarrow{G^{'m_1} P_j'}}{|\overrightarrow{G^{'m_1} P_i'} \times \overrightarrow{G^{'m_1} P_j'}|} \cdot \dfrac{\overrightarrow{G^{'m_2} P_j'} \times \overrightarrow{G^{'m_2} P_i'}}{|\overrightarrow{G^{'m_2} P_j'} \times \overrightarrow{G^{'m_2} P_i'}|}\right)\\
		&\approx \dfrac{\left|(\overrightarrow{G^{m_1} P_i} \times \overrightarrow{G^{m_1} P_j}) \times (\overrightarrow{G^{m_2} P_j}\times \overrightarrow{G^{m_2} P_i})\right|}{|\overrightarrow{G^{m_2} P_j} \times \overrightarrow{G^{m_2} P_i}|^4} \Bigg[\left((\overrightarrow{G^{m_1} P_i} \times \overrightarrow{G^{m_1} P_j}) \cdot (\overrightarrow{G^{m_2} P_j} \times \overrightarrow{G^{m_2} P_i})\right)\\
		&\quad-\left((\overrightarrow{G^{'m_1} P_i'} \times \overrightarrow{G^{'m_1} P_j'}) \cdot (\overrightarrow{G^{'m_2} P_j'} \times \overrightarrow{G^{'m_2} P_i'}) \right) \Bigg].
	\end{aligned}
\end{equation}

Let us study the term
$$
(\overrightarrow{G^{'m_1} P_i'} \times \overrightarrow{G^{'m_1} P_j'}) \cdot (\overrightarrow{G^{'m_2} P_j'} \times \overrightarrow{G^{'m_2} P_i'}),
$$
in the previous set of equations~\eqref{angle:1}. We have:
\begin{align*}
	\!\!\!\!\!\!\!\!\!&(\overrightarrow{G^{'m_1} P_i'} \times \overrightarrow{G^{'m_1} P_j'}) \cdot (\overrightarrow{G^{'m_2} P_j'} \times \overrightarrow{G^{'m_2} P_i'})\\
	&=\left((\overrightarrow{O P_i'}-\overrightarrow{OG^{'m_1}}) \times (\overrightarrow{O P_j'}-\overrightarrow{O G^{'m_1}})\right) \cdot \left((\overrightarrow{O P_j'}-\overrightarrow{OG^{'m_2}})\times(\overrightarrow{O P_i'}-\overrightarrow{OG^{'m_2}})\right)\\
	&=\Bigg[\left(\overrightarrow{O P_i}+\vec{u}_i-\overrightarrow{OG^{m_1}}-\sum_{r \in \mathscr{M}(m_1)}\dfrac{\vec{u}_r}{3}\right)
	\times \left(\overrightarrow{O P_j}+\vec{u}_j-\overrightarrow{OG^{m_1}}-\sum_{r \in \mathscr{M}(m_1)}\dfrac{\vec{u}_r}{3}\right)\Bigg]\\
	&\quad \cdot\Bigg[\left(\overrightarrow{O P_j}+\vec{u}_j-\overrightarrow{OG^{m_2}}-\sum_{r \in \mathscr{M}(m_2)}\dfrac{\vec{u}_r}{3}\right)\times \left(\overrightarrow{O P_i}+\vec{u}_i-\overrightarrow{OG^{m_2}}-\sum_{r \in \mathscr{M}(m_2)}\dfrac{\vec{u}_r}{3}\right)\Bigg]\\
	&=\Bigg[\left(\overrightarrow{O P_i}-\overrightarrow{O G^{m_1}}+\dfrac{2}{3}\vec{u}_i-\sum_{\substack{r \in \mathscr{M}(m_1)\\r \neq i}}\dfrac{\vec{u}_r}{3}\right)
	\times \left(\overrightarrow{O P_j}-\overrightarrow{O G^{m_1}}+\dfrac{2}{3}\vec{u}_j-\sum_{\substack{r \in \mathscr{M}(m_1)\\r \neq j}}\dfrac{\vec{u}_r}{3}\right)\Bigg]\\
	&\quad \cdot \Bigg[\left(\overrightarrow{O P_j}-\overrightarrow{O G^{m_2}}+\dfrac{2}{3}\vec{u}_j-\sum_{\substack{r \in \mathscr{M}(m_2)\\r \neq j}}\dfrac{\vec{u}_r}{3}\right)
	\times \left(\overrightarrow{O P_i}-\overrightarrow{O G^{m_2}}+\dfrac{2}{3}\vec{u}_i-\sum_{\substack{r \in \mathscr{M}(m_2)\\r \neq i}}\dfrac{\vec{u}_r}{3}\right)\Bigg]\\
	&\approx \Bigg[(\overrightarrow{O P_i}-\overrightarrow{O G^{m_1}}) \times (\overrightarrow{O P_j}-\overrightarrow{O G^{m_1}})
	+(\overrightarrow{O P_i}-\overrightarrow{O G^{m_1}}) \times \left(\dfrac{2}{3} \vec{u}_j -\sum_{\substack{r \in \mathscr{M}(m_1)\\r \neq j}} \dfrac{\vec{u}_r}{3}\right)\\
	&\quad-(\overrightarrow{O P_j}-\overrightarrow{O G^{m_1}})\times \left(\dfrac{2}{3} \vec{u}_i -\sum_{\substack{r \in \mathscr{M}(m_1)\\r \neq i}} \dfrac{\vec{u}_r}{3}\right)\Bigg]\\
	&\quad \cdot \Bigg[(\overrightarrow{O P_j}-\overrightarrow{O G^{m_2}}) \times (\overrightarrow{O P_i}-\overrightarrow{O G^{m_2}})+(\overrightarrow{O P_j}-\overrightarrow{O G^{m_2}}) \times \left(\dfrac{2}{3} \vec{u}_i -\sum_{\substack{r \in \mathscr{M}(m_2)\\r \neq i}} \dfrac{\vec{u}_r}{3}\right)\\
	&\quad -(\overrightarrow{O P_i}-\overrightarrow{O G^{m_2}}) \times \left(\dfrac{2}{3} \vec{u}_j -\sum_{\substack{r \in \mathscr{M}(m_2)\\r \neq j}} \dfrac{\vec{u}_r}{3}\right)\Bigg]\\
	&\approx \left[(\overrightarrow{O P_i}-\overrightarrow{O G^{m_1}})\times(\overrightarrow{O P_j}-\overrightarrow{O G^{m_1}})\right] \cdot \left[(\overrightarrow{O P_j}-\overrightarrow{O G^{m_2}})\times(\overrightarrow{O P_i}-\overrightarrow{O G^{m_2}})\right]\\
	&\quad+\left((\overrightarrow{O P_i}-\overrightarrow{O G^{m_1}})\times(\overrightarrow{O P_j}-\overrightarrow{O G^{m_1}})\right) \cdot
	\left((\overrightarrow{O P_j}-\overrightarrow{O G^{m_2}})\times\left(\dfrac{2}{3} \vec{u}_i -\sum_{\substack{r \in \mathscr{M}(m_2)\\r \neq i}} \dfrac{\vec{u}_r}{3}\right)\right)\\
	&\quad-\left((\overrightarrow{O P_i}-\overrightarrow{O G^{m_1}})\times(\overrightarrow{O P_j}-\overrightarrow{O G^{m_1}})\right) \cdot
	\left((\overrightarrow{O P_i}-\overrightarrow{O G^{m_2}})\times\left(\dfrac{2}{3} \vec{u}_j -\sum_{\substack{r \in \mathscr{M}(m_2)\\r \neq j}} \dfrac{\vec{u}_r}{3}\right)\right)\\
	&\quad +\left((\overrightarrow{O P_i}-\overrightarrow{O G^{m_1}}) \times \left(\dfrac{2}{3}\vec{u}_j-\sum_{\substack{r \in \mathscr{M}(m_1)\\r \neq j}}\dfrac{\vec{u}_r}{3}\right)\right)\cdot\left((\overrightarrow{O P_j}-\overrightarrow{O G^{m_2}})\times(\overrightarrow{O P_i}-\overrightarrow{O G^{m_2}})\right)\\
	&\quad- \left((\overrightarrow{O P_j}-\overrightarrow{O G^{m_1}}) \times \left(\dfrac{2}{3} \vec{u}_i -\sum_{\substack{r \in \mathscr{M}(m_1)\\r \neq i}} \dfrac{\vec{u}_r}{3}\right)\right) \cdot \left((\overrightarrow{O P_j}-\overrightarrow{O G^{m_2}}) \times (\overrightarrow{O P_i}-\overrightarrow{O G^{m_2}})\right)\\
	&=(\overrightarrow{G^{m_1}P_i} \times \overrightarrow{G^{m_1}P_j}) \cdot (\overrightarrow{G^{m_2}P_j}\times\overrightarrow{G^{m_2}P_i})\\
	&\quad+(\overrightarrow{G^{m_1}P_i} \times \overrightarrow{G^{m_1}P_j}) \cdot \left[\overrightarrow{G^{m_2}P_j}\times \left(\dfrac{2}{3}\vec{u}_i+\sum_{\substack{r \in \mathscr{M}(m_2)\\r \neq i}}\dfrac{\vec{u}_r}{3}\right)\right]\\
	&\quad-(\overrightarrow{G^{m_1}P_i} \times \overrightarrow{G^{m_1}P_j}) \cdot \left[\overrightarrow{G^{m_2}P_i}\times \left(\dfrac{2}{3}\vec{u}_j+\sum_{\substack{r \in \mathscr{M}(m_2)\\r \neq j}}\dfrac{\vec{u}_r}{3}\right)\right]\\
	&\quad+\left[\overrightarrow{G^{m_1}P_i} \times \left(\dfrac{2}{3}\vec{u}_j+\sum_{\substack{r \in \mathscr{M}(m_1)\\r \neq j}}\dfrac{\vec{u}_r}{3}\right)\right]
	\cdot (\overrightarrow{G^{m_2}P_j} \times \overrightarrow{G^{m_2}P_i})\\
	&\quad-\left[\overrightarrow{G^{m_1}P_j} \times \left(\dfrac{2}{3}\vec{u}_i+\sum_{\substack{r \in \mathscr{M}(m_1)\\r \neq i}}\dfrac{\vec{u}_r}{3}\right)\right]
	\cdot (\overrightarrow{G^{m_2}P_j} \times \overrightarrow{G^{m_2}P_i}).
\end{align*}

As a result, we have that:

\begin{align*}
	|\alpha-\alpha'| & \approx \dfrac{\left|(\overrightarrow{G^{m_1} P_i} \times \overrightarrow{G^{m_1} P_j}) \times (\overrightarrow{G^{m_2} P_j}\times \overrightarrow{G^{m_2} P_i})\right|}{|\overrightarrow{G^{m_2} P_j} \times \overrightarrow{G^{m_2} P_i}|^4} \\
	&\quad \cdot \Bigg|((\overrightarrow{G^{m_1}P_i} \times \overrightarrow{G^{m_1}P_j})\times \overrightarrow{G^{m_2}P_j}) \cdot \left(\dfrac{2}{3}\vec{u}_i+\sum_{\substack{r \in \mathscr{M}(m_2)\\r \neq i}}\dfrac{\vec{u}_r}{3}\right)\\
	&\quad-((\overrightarrow{G^{m_1}P_i} \times \overrightarrow{G^{m_1}P_j})\times \overrightarrow{G^{m_2}P_i}) \cdot \left(\dfrac{2}{3}\vec{u}_j+\sum_{\substack{r \in \mathscr{M}(m_2)\\r \neq j}}\dfrac{\vec{u}_r}{3}\right)\\
	&\quad+((\overrightarrow{G^{m_2}P_j} \times \overrightarrow{G^{m_2}P_i})\times\overrightarrow{G^{m_1}P_i})\cdot\left(\dfrac{2}{3}\vec{u}_j+\sum_{\substack{r \in \mathscr{M}(m_1)\\r \neq j}}\dfrac{\vec{u}_r}{3}\right)\\
	&\quad-((\overrightarrow{G^{m_2}P_j} \times \overrightarrow{G^{m_2}P_i})\times\overrightarrow{G^{m_1}P_j})\cdot\left(\dfrac{2}{3}\vec{u}_i+\sum_{\substack{r \in \mathscr{M}(m_1)\\r \neq i}}\dfrac{\vec{u}_r}{3}\right)\Bigg|=:J_{m_1,m_2}.
\end{align*}

Observe that the constant 
$$
C:=\dfrac{\left|(\overrightarrow{G^{m_1} P_i} \times \overrightarrow{G^{m_1} P_j}) \times (\overrightarrow{G^{m_2} P_j}\times \overrightarrow{G^{m_2} P_i})\right|}{|\overrightarrow{G^{m_2} P_j} \times \overrightarrow{G^{m_2} P_i}|^4}
$$
is uniform with respect to the indices since the edges of the polyhedron under consideration all have the same length.

The total bending energy can thus be expressed in terms of the vertices displacements:
\begin{equation*}
	J_b(\bm{U})=\dfrac{k_b C^2}{2} \sum_{m_1=1}^{20}  \sum_{\substack{m_2 \neq m_1\\|\overrightarrow{G^{m_1}G^{m_2}}|=g}}|J_{m_1,m_2}|^2.
\end{equation*}

It is immediate to observe that the functional $J_b$ is a convex function of $\bm{U}$, and hence that $J_s(\bm{U}) +J_b(\bm{U})$ is strictly convex.

Finally, the last component of the total elastic energy is given by the variation of the Gaussian curvature, which takes the following form (cf., e.g., \cite{Helfrich1973})
$$
J_G=\dfrac{k_G}{2}\sum_{i=0}^{11} (K(P_i)-K(P_i')),
$$
where $k_G>0$ is the elastic modulus associated with the Gaussian curvature.

If the deformation is small enough, it is licit to \emph{assume} that the deformed reference configuration is still a convex polyhedron.
As a result, we have that
\begin{align*}
	J_G&=\dfrac{k_G}{2}\sum_{i=0}^{11} (K(P_i)-K(P_i'))=\dfrac{k_G}{2}\left\{\sum_{i=0}^{11}K(P_i) - \sum_{i=0}^{11}K(P_i')\right\}\\
	&=\dfrac{k_G}{2}(4 \pi - 4\pi)=0,
\end{align*}
where the second last equality holds thanks to \emph{the Descartes Lost Theorem} (Theorem~\ref{Descartes}). In conclusion, the energetic contribution due to the variation of the Gaussian curvature is \emph{exactly} equal to zero.

In conclusion, the total elastic energy associated with the displacement tensor $\bm{U}$ takes the following form:
\begin{equation*}
	J(\bm{U})=J_s(\bm{U}) +J_b(\bm{U}).
\end{equation*}

The functional $J$ is strictly convex since $J_s$ is strictly convex, and $J_b$ is convex.

The search for an equilibrium position for the deformed polygon amounts to minimizing the total elastic energy functional $J$.
In view of the geometrical constraint according to which the vertices must not cross the given flat surface, the admissible displacement fields are to be sought in the following set
\begin{align*}
	\mathcal{U}:=\bigg\{\bm{V}=(\vec{v}_i)_{i=1}^{11} \in \mathbb{R}^{33}; &\vec{v}_i=\begin{pmatrix}v_{i,1},v_{i,2},v_{i,3}\end{pmatrix}^T \textup{ and }\\
	&(\overrightarrow{OP_i}+\vec{v}_i)\cdot \vec{e}_3 \ge 0 \textup{ for all }1 \le i \le 11\bigg\}.
\end{align*}

It is straightforward to observe that the set $\mathcal{U}$ is non-empty (as $\bm{V}=\bm{0}\in \mathbb{R}^{33}\in \mathcal{U}$), closed, and convex. Recall that $\bm{F}=(\vec{f}_i)_{i=1}^{11} \in \mathbb{R}^{33}$ denotes the array of applied body forces acting on the regular icosahedron vertices (cf. at the beginning of Section~\ref{Sec:2}).

Therefore, the latter together with the fact that the total elastic energy functional is strictly convex, imply that the quadratic minimization problem 
\begin{equation*}
	\inf_{\bm{V} \in \mathcal{U}} (J(\bm{V}) -\bm{F} \cdot \bm{V})
\end{equation*}
admits a unique minimizer (cf., e.g., Proposition~1.2 of~\cite{EkelandTemam1999}). Finding the solution for this minimization problem is equivalent to finding a tensor $\bm{U}$ that solves the following variational inequalities:
\begin{equation*}
	\left(\left(\dfrac{k_s}{2} \Sigma^T \Sigma+\dfrac{k_b C^2}{2} \Theta^T \Theta\right) \bm{U}\right) \cdot (\bm{V}-\bm{U}) \ge \bm{F} \cdot (\bm{V}-\bm{U}),
	\quad\textup{ for all } \bm{V} \in \mathcal{U}.
\end{equation*}

Both the matrices $\Sigma$ and $\Theta$ have 20 rows and 33 columns. They are associated with the formulation of the stretching and bending energy, respectively.
Given the high number of variables entering the problem, it is not easy to explicitly write down these matrices without the  aid of computing software, as it was for the two-dimensional simplified case treated in the paper~\cite{PieWhiDraTem2021}. The numerical experiments that we will implement in the forthcoming sections make use of the matrix formulation for the variational inequalities stated above.

Since the functional $J_s$ is strictly convex, it is clear on the one hand that the square matrix $\Sigma^T \Sigma$, which counts 33 rows and 33 columns, is strictly positive-definite (in the sense that the smallest eigenvalue is greater than zero). On the other hand, the square matrix $\Theta^T \Theta$ is a priori only non-negative definite (in the sense that the smallest eigenvalue is greater or equal than zero). With the aid of computing software, it can be indeed verified that the determinant of the non-negative definite matrix $\Theta^T \Theta$ is of the order $10^{-231}$, which is zero in view of the precision of the calculations.

This lets us infer that there cannot be bending without stretching. The latter statement makes sense from both the physical and geometrical points of view, since it is not possible to change the inclination of the faces of a regular icosahedron without changing the lengths of its edges.

\section{Numerical experiments}
\label{Sec:3}

In this section we report on two batches of numerical experiments intended to test the model presented in Section~\ref{Sec:3}.

We consider different instances of the array  of applied body forces whose tangential components are equal zero and whose transverse component are directed downwards.
The values for the elastic constants are $k_s=0.25$ and $k_b=1.7$. We assume the length of each edge $\ell$ to be equal to $3$.

The first batch of numerical experiments that we conducted on the proposed model is classical, and amounts to finding the position of the deformed reference configuration of the icosahedron undergoing the action of an applied body force $\bm{F}=(\vec{f}_i)_{i=1}^{11}$ which acts on each vertex with the same magnitude. 
Recall that the undeformed reference configuration of the icosahedral cage is in contact with the obstacle at one point, and that this point is not displaced.


In the second batch of experiments, the force acting on the icosahedral cage is released, and we compute the equilibrium position of the cage under the constraint that the points which are in contact with the obstacle at the end of the first experiment continue remaining in contact with the obstacle for the whole duration of the second experiment.

The two experiments are summarized in the diagrams below:

\begin{algorithm}[H]
	\caption{}
	\begin{algorithmic}[!H]
		\Input
		\Desc{Undeformed reference configuration of the icosahedral viral cage as in Figure~\ref{fig:0}}
		\Desc{Applied body force $\bm{F}=(\vec{f}_i)_{i=1}^{11}$}
		\EndInput
		\\
		\State Compute the deformation associated with the input $\bm{F}$ via the primal-dual active set method~\cite{SunYuan2006}
		\\
		\Output
		\Desc{Deformed reference configuration of the icosahedral viral cage}
		\EndOutput
	\end{algorithmic}
	\label{algo1}
\end{algorithm}

\begin{algorithm}[H]
	\caption{}
	\begin{algorithmic}[!H]
		\Input
		\Desc{Deformed reference configuration of the icosahedral viral cage (Output of Algorithm~\ref{algo1})}
		\EndInput
		\\
		\State Compute the equilibrium configuration subjected to the constraint that points in the input that are in contact with the obstacle continue remaining in contact with the obstacle 
		\\
		\Output
		\Desc{Equilibrium shape of the icosahedral viral cage with points in contact}
		\EndOutput
	\end{algorithmic}
	\label{algo2}
\end{algorithm}

In Figure~\ref{fig:3} below, each row displays the output of the two batches of numerical experiments described beforehand. For each row, starting from the left, the first two figures illustrate the deformed reference configuration, seen from the top and from the side, respectively, of the icosahedral cage when it undergoes the action of a vertical applied body force $\bm{F}=(\vec{f}_i)_{i=1}^{11}$. The third figure represents the equilibrium shape recovered by releasing the force acting on the deformed icosahedral cage, and by minimizing the total energy functional $J$ subjected to the constraint that the points of the deformed reference configuration that are in contact with the obstacle at the end of the first experiment continue remaining in contact with the obstacle for the whole duration of the second experiment.

First, we observe that, as the magnitude of the applied body forces acting on the vertices of the icosahedron increases, the height of the top point of the reference configuration decreases.

Secondly, we observe that, as expected, the equilibrium shape corresponding to two different references configurations which have the same number of points in contact with the obstacle is the same.

\begin{figure}[H]
	\centering
	\makebox[0pt][c]{
	\begin{subfigure}[b]{0.4\linewidth}
		\includegraphics[width=1.0\linewidth]{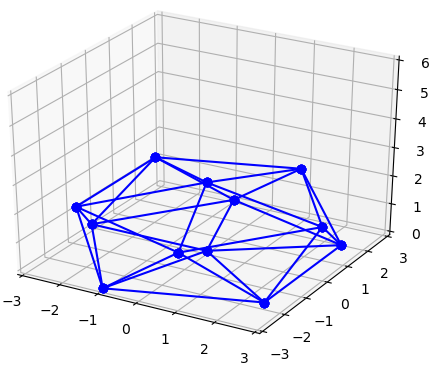}
		\subcaption{\scriptsize Deformation of the viral capsid seen from above for $f_{i,3}=-5.3$}
		\label{sub:1:1}
	\end{subfigure}%
	\hspace{0.5cm}
	\begin{subfigure}[b]{0.4\linewidth}
		\includegraphics[width=1.0\linewidth]{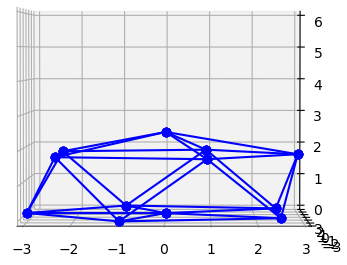}
		\subcaption{\scriptsize Deformation of the viral capsid seen from the side for $f_{i,3}=-5.3$}
		\label{sub:1:2}
	\end{subfigure}%
	\hspace{0.5cm}
	\begin{subfigure}[b]{0.4\linewidth}
		\includegraphics[width=1.0\linewidth]{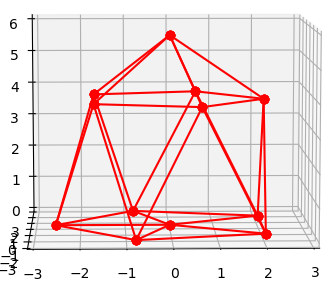}
		\subcaption{\scriptsize Equilibrium shape of the viral capsid corresponding to $f_{i,3}=-5.3$}
		\label{sub:1:3}
	\end{subfigure}%
	}
	\caption{Deformations of a regular icosahedron corresponding to a purely vertical applied body force such that $\vec{f}_i=(0,0,-5.3)$ for all $1\le i \le 11$, and restoration of the equilibrium position. Figures~\ref{sub:1:1} and~\ref{sub:1:2} depict the deformed reference configuration output by an implementation of Algorithm~\ref{algo1} seen from the above and from the side, respectively. Figure~\ref{sub:1:3}, instead, depicts the output of Algorithm~\ref{algo2}. In this case equilibrium configuration associated with the displacement field that minimizes the total elastic energy subjected to the constraint that the points that are in contact with the obstacle.}
\end{figure}

\begin{figure}[H]
	\centering
	\makebox[0pt][c]{
	\begin{subfigure}[b]{0.4\linewidth}\ContinuedFloat
		\includegraphics[width=1.0\linewidth]{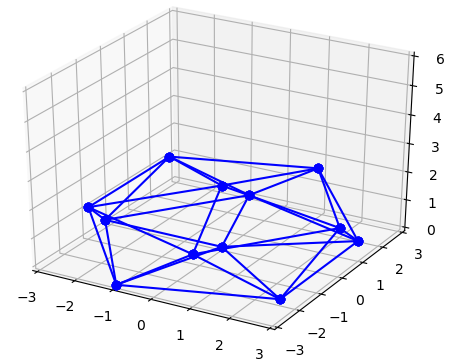}
		\subcaption{Deformation of the viral capsid seen from above for $f_{i,3}=-5.7$}
		\label{sub:2:1}
	\end{subfigure}%
	\hspace{0.5cm}
	\begin{subfigure}[b]{0.4\linewidth}\ContinuedFloat
		\includegraphics[width=1.0\linewidth]{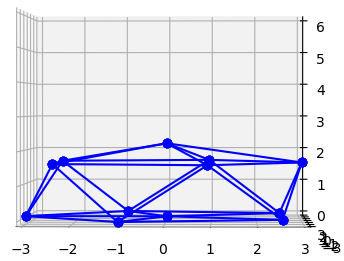}
		\subcaption{Deformation of the viral capsid seen from the side for $f_{i,3}=-5.7$}
		\label{sub:2:2}
	\end{subfigure}%
	\hspace{0.5cm}
	\begin{subfigure}[b]{0.4\linewidth}\ContinuedFloat
		\includegraphics[width=1.0\linewidth]{./Figures/Equilibrium}
		\subcaption{Equilibrium shape of the viral capsid corresponding to $f_{i,3}=-5.7$}
		\label{sub:2:3}
	\end{subfigure}%
	}
	\caption{Deformations of a regular icosahedron corresponding to a purely vertical applied body force such that $\vec{f}_i=(0,0,-5.7)$ for all $1\le i \le 11$, and restoration of the equilibrium position. Figures~\ref{sub:2:1} and~\ref{sub:2:2} depict the deformed reference configuration output by an implementation of Algorithm~\ref{algo1} seen from the above and from the side, respectively. Figure~\ref{sub:2:3}, instead, depicts the output of Algorithm~\ref{algo2}. In this case equilibrium configuration associated with the displacement field that minimizes the total elastic energy subjected to the constraint that the points that are in contact with the obstacle.}
\end{figure}

\begin{figure}[H]
	\centering
	\makebox[0pt][c]{
	\begin{subfigure}[b]{0.4\linewidth}
		\includegraphics[width=1.0\linewidth]{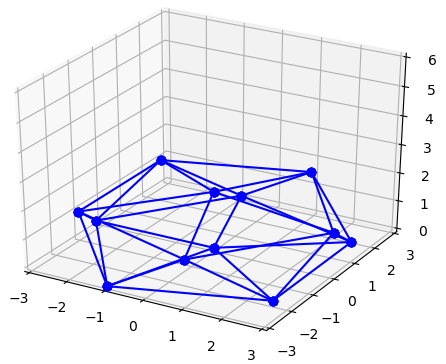}
		\subcaption{Deformation of the viral capsid seen from above for $f_{i,3}=-6.0$}
		\label{sub:3:1}
	\end{subfigure}%
	\hspace{0.5cm}
	\begin{subfigure}[b]{0.4\linewidth}\ContinuedFloat
		\includegraphics[width=1.0\linewidth]{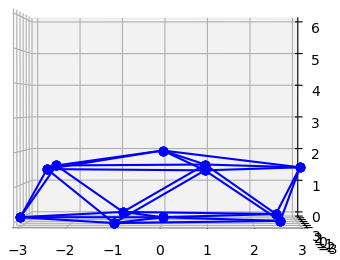}
		\subcaption{Deformation of the viral capsid seen from the side for $f_{i,3}=-6.0$}
		\label{sub:3:2}
	\end{subfigure}%
	\hspace{0.5cm}
	\begin{subfigure}[b]{0.4\linewidth}\ContinuedFloat
		\includegraphics[width=1.0\linewidth]{./Figures/Equilibrium}
		\subcaption{Equilibrium shape of the viral capsid corresponding to $f_{i,3}=-6.0$}
		\label{sub:3:3}
	\end{subfigure}%
	}
	\caption{Deformations of a regular icosahedron corresponding to a purely vertical applied body force such that $\vec{f}_i=(0,0,-6.0)$ for all $1\le i \le 11$, and restoration of the equilibrium position. Figures~\ref{sub:3:1} and~\ref{sub:3:2} depict the deformed reference configuration output by an implementation of Algorithm~\ref{algo1} seen from the above and from the side, respectively. Figure~\ref{sub:3:3}, instead, depicts the output of Algorithm~\ref{algo2}. In this case equilibrium configuration associated with the displacement field that minimizes the total elastic energy subjected to the constraint that the points that are in contact with the obstacle.}
\end{figure}

\begin{figure}[H]
	\centering
	\makebox[0pt][c]{
	\begin{subfigure}[b]{0.4\linewidth}
		\includegraphics[width=1.0\linewidth]{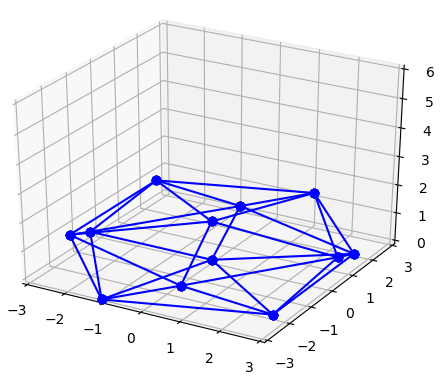}
		\subcaption{Deformation of the viral capsid seen from above for $f_{i,3}=-7.0$}
		\label{sub:4:1}
	\end{subfigure}%
	\hspace{0.5cm}
	\begin{subfigure}[b]{0.4\linewidth}\ContinuedFloat
		\includegraphics[width=1.0\linewidth]{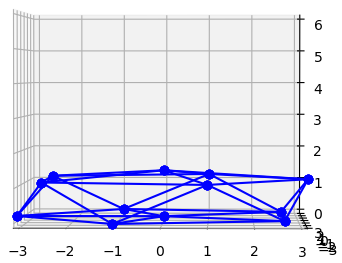}
		\subcaption{Deformation of the viral capsid seen from the side for $f_{i,3}=-7.0$}
		\label{sub:4:2}
	\end{subfigure}%
	\hspace{0.5cm}
	\begin{subfigure}[b]{0.4\linewidth}\ContinuedFloat
		\includegraphics[width=1.0\linewidth]{./Figures/Equilibrium}
		\subcaption{Equilibrium shape of the viral capsid corresponding to $f_{i,3}=-7.0$}
		\label{sub:4:3}
	\end{subfigure}%
	}
	\caption{Deformations of a regular icosahedron corresponding to a purely vertical applied body force such that $\vec{f}_i=(0,0,-7.0)$ for all $1\le i \le 11$, and restoration of the equilibrium position. Figures~\ref{sub:4:1} and~\ref{sub:4:2} depict the deformed reference configuration output by an implementation of Algorithm~\ref{algo1} seen from the above and from the side, respectively. Figure~\ref{sub:4:3}, instead, depicts the output of Algorithm~\ref{algo2}. In this case equilibrium configuration associated with the displacement field that minimizes the total elastic energy subjected to the constraint that the points that are in contact with the obstacle.}
	\label{fig:3}
\end{figure}

\section{Deformation of viral capsids subjected to a confinement condition in the time-dependent case}
\label{Sec:4}

In this section we present a time-dependent model describing the evolution of the deformation of a viral capsid subjected to the action of an applied body force, and confined not to cross a given rigid surface on which it lies in correspondence of one point at the beginning of the observation. Moreover, we recall that, experimental data led us to assume the contact point does not change its position during the deformation.
This means that $\overrightarrow{O P_i}\cdot\vec{e}_3 > 0$ for all $1\le i \le 11$.
We limit ourselves to observing the deformation in a finite length time interval of the form $[0,T]$, with $T>0$ given.

The model is hyperbolic, since it takes into account the evolution of the displacement. The formulation of the corresponding variational problem and the concept of solutions are, however, \emph{a priori} not clear. For this reason, we will recover the rigorous concept of solution for the problem under consideration and we will prove the existence of solutions for one such model by resorting to a technique based on the penalty method which was originally developed for the continuum case in the paper~\cite{BockJarSil2016}. 


By virtue of the physical model under consideration, we require the displacement $\bm{U}=(\vec{u}_i)_{i=1}^{11}:[0,T] \to \mathbb{R}^{33}$ to be such that $(\overrightarrow{O P_i}+\vec{u}_i(t)) \cdot \vec{e}_3 \ge 0$, for all, or almost all (a.a. in what follows) $t \in [0,T]$.

Define the set 
$$
K:=\{\bm{V}=(\vec{v}_i)_{i=1}^{11} \in \mathbb{R}^{33}; (\overrightarrow{O P_i} +\vec{v}_i)\cdot \vec{e}_3 \ge 0 \textup{ for all }1 \le i \le 11\},
$$
and define the set
$$
\mathcal{K}:=\{\bm{V}=(\vec{v}_i)_{i=1}^{11} \in \mathcal{C}^0([0,T];\mathbb{R}^{33}); \bm{V}(t) \in K \textup{ for a.a. }t\in (0,T)\}.
$$

Define the linear operator $\gamma: L^2(0,T;\mathbb{R}^3) \to L^2(0,T)$ in a way such that 
$$
\gamma(\vec{v})(t):=\vec{v}(t) \cdot \vec{e}_3 \textup{ for a.a. } t \in (0,T).
$$

The operator $\gamma$ is clearly bounded by the Cauchy-Schwarz inequality, and thus continuous. As a result, the operator $\gamma$ admits a uniquely determined Hilbert adjoint, which we denote by $\gamma^\ast$ (cf., e.g., Theorem~4.7-2 of~\cite{PGCLNFAA}).

For all $1\le i \le 11$, define the operator $\beta_i:L^2(0,T) \to L^2(0,T)$ by
$$
\beta_i(f):=-\{\overrightarrow{O P_i} \cdot\vec{e}_3 +f\}^{-},\quad\textup{ for all }f \in L^2(0,T).
$$

The following properties for the operators $\beta_i$ can be established.

\begin{theorem}
\label{betai}
For all $1\le i \le 11$, the operator $\beta_i$ is monotone, bounded and Lipschitz continuous with Lipschitz constant $L=1$.
\end{theorem}
\begin{proof}
Fix $1\le i \le 11$.
For the sake of brevity, sets of the form $\{f \lesssim 0 \, \& \, g \lesssim 0\}$ denote the sets:
$$
\{t\in (0,T); \overrightarrow{O P_i}\cdot \vec{e}_3 + f(t) \le 0\} \cap \{t\in(0,T);\overrightarrow{O P_i}\cdot \vec{e}_3 + g(t) \le 0\}.
$$

For proving the monotonicity, we observe that for all $f,g \in L^2(0,T)$ we have:
\begin{equation*}
\begin{aligned}
&\int_{0}^{T} (\beta_i(f)-\beta_i(g)) (f-g) \dd t\\
=&\int_{0}^{T} \left((-\{\overrightarrow{O P_i} \cdot\vec{e}_3 +f\}^{-})-(-\{\overrightarrow{O P_i} \cdot\vec{e}_3 +g\}^{-})\right) \left((\overrightarrow{O P_i} \cdot\vec{e}_3 +f)-(\overrightarrow{O P_i} \cdot\vec{e}_3 +g)\right) \dd t\\
=&\int_{0}^{T} \left|-\{\overrightarrow{O P_i} \cdot\vec{e}_3 +f\}^{-}\right|^2 \dd t + \int_{0}^{T} \left|-\{\overrightarrow{O P_i} \cdot\vec{e}_3 +g\}^{-}\right|^2 \dd t\\
&\quad -\int_{0}^{T} \left(-\{\overrightarrow{O P_i} \cdot\vec{e}_3 +f\}^{-}\right) \left(\overrightarrow{O P_i} \cdot\vec{e}_3 +g\right) \dd t\\
&\quad-\int_{0}^{T} \left(-\{\overrightarrow{O P_i} \cdot\vec{e}_3 +g\}^{-}\right) \left(\overrightarrow{O P_i} \cdot\vec{e}_3 +f\right) \dd t\\
=&\int_{0}^{T} \left|-\{\overrightarrow{O P_i} \cdot\vec{e}_3 +f\}^{-}\right|^2 \dd t + \int_{0}^{T} \left|-\{\overrightarrow{O P_i} \cdot\vec{e}_3 +g\}^{-}\right|^2 \dd t\\
&\quad -\int_{\{f \lesssim 0 \, \& \, g \lesssim 0\}} \left(-\{\overrightarrow{O P_i} \cdot\vec{e}_3 +f\}^{-}\right) \left(-\{\overrightarrow{O P_i} \cdot\vec{e}_3 +g\}^{-}\right) \dd t\\
&\quad -\int_{\{f \lesssim 0 \, \& \, g \gtrsim 0\}} \left(-\{\overrightarrow{O P_i} \cdot\vec{e}_3 +f\}^{-}\right) \left(\{\overrightarrow{O P_i} \cdot\vec{e}_3 +g\}^{+}\right) \dd t\\
&\quad-\int_{\{f \lesssim 0 \, \& \, g \lesssim 0\}} \left(-\{\overrightarrow{O P_i} \cdot\vec{e}_3 +g\}^{-}\right) \left(-\{\overrightarrow{O P_i} \cdot\vec{e}_3 +f\}^{-}\right) \dd t\\
&\quad-\int_{\{f \gtrsim 0 \, \& \, g \lesssim 0\}} \left(-\{\overrightarrow{O P_i} \cdot\vec{e}_3 +g\}^{-}\right) \left(\{\overrightarrow{O P_i} \cdot\vec{e}_3 +f\}^{+}\right) \dd t\\
&\ge \int_{\{f \lesssim 0 \, \& \, g \lesssim 0\}} \left|-\{\overrightarrow{O P_i} \cdot\vec{e}_3 +f\}^{-}\right|^2 \dd t + \int_{\{f \lesssim 0 \, \& \, g \lesssim 0\}} \left|-\{\overrightarrow{O P_i} \cdot\vec{e}_3 +g\}^{-}\right|^2 \dd t\\
&\quad-2\int_{\{f \lesssim 0 \, \& \, g \lesssim 0\}} \left(-\{\overrightarrow{O P_i} \cdot\vec{e}_3 +f\}^{-}\right) \left(-\{\overrightarrow{O P_i} \cdot\vec{e}_3 +g\}^{-}\right) \dd t\\
=&\int_{\{f \lesssim 0 \, \& \, g \lesssim 0\}} \left|\left(-\{\overrightarrow{O P_i} \cdot\vec{e}_3 +f\}^{-}\right)-\left(-\{\overrightarrow{O P_i} \cdot\vec{e}_3 +g\}^{-}\right)\right|^2 \dd t \ge0.
\end{aligned}
\end{equation*}

For establishing the boundedness, we show that each nonlinear mapping $\beta_i$ maps bounded sets onto bounded sets. To see this, let $\mathcal{F} \subset L^2(0,T)$ be a bounded subset and evaluate, for all $1 \le i \le 11$ and all $f \in \mathcal{F}$,
\begin{align*}
\|\beta_i(f)\|_{L^2(0,T)}&=\sup_{\substack{g \in L^2(0,T)\\\|g\|_{L^2(0,T)}=1}} \left|\int_{0}^{T} \beta_i(f) g \dd t\right| \le \|\{\overrightarrow{O P_i}\cdot \vec{e}_3+f\}^{-}\|_{L^2(0,T)}\\
&\le T \max_{1 \le i \le 11}\{\overrightarrow{O P_i} \cdot \vec{e}_3\} +\|f\|_{L^2(0,T)},
\end{align*}
where the second last inequality is obtained as a result of an application of the Cauchy-Schwarz inequality.

The boundedness of the family $\mathcal{F}$ and the uniform boundedness of $\max_{1 \le i \le 11}\{\overrightarrow{O P_i} \cdot \vec{e}_3\}$ in turn imply that there actually exists a constant $c>0$ independent of $i$ for which
$$
\max_{1 \le i \le 11} \|\beta_i(f)\|_{L^2(0,T)} \le c,
$$
thus establishing the boundedness for all the functions $\beta_i$ by means of the same uniform constant.

Finally, in order to establish the Lipschitz continuity, for each $1 \le i \le 11$ we compute
\begin{align*}
&\left(\int_{0}^{T} |\beta_i(f)-\beta_i(g)|^2 \dd t\right)^{1/2} = \left(\int_{0}^{T} \left|\left(-\{\overrightarrow{O P_i} \cdot\vec{e}_3 +f\}^{-}\right)-\left(-\{\overrightarrow{O P_i} \cdot\vec{e}_3 +g\}^{-}\right)\right|^2 \dd t\right)^{1/2}\\
&=\left(\int_{0}^{T}\left|\dfrac{(\overrightarrow{O P_i} \cdot\vec{e}_3 +f)-|\overrightarrow{O P_i} \cdot\vec{e}_3 +f|}{2} - \dfrac{(\overrightarrow{O P_i} \cdot\vec{e}_3 +g)-|\overrightarrow{O P_i} \cdot\vec{e}_3 +g|}{2}\right|^2 \dd t\right)^{1/2}\\
&=\left(\int_{0}^{T}\left|\dfrac{(\overrightarrow{O P_i} \cdot\vec{e}_3 +f)-(\overrightarrow{O P_i} \cdot\vec{e}_3 +g)}{2} - \dfrac{|\overrightarrow{O P_i} \cdot\vec{e}_3 +f|-|\overrightarrow{O P_i} \cdot\vec{e}_3 +g|}{2}\right|^2 \dd t\right)^{1/2}\\
&\le \dfrac{1}{2}\|f-g\|_{L^2(0,T)} +\dfrac{1}{2}\left(\int_{0}^{T}\left||\overrightarrow{O P_i} \cdot\vec{e}_3 +f|-|\overrightarrow{O P_i} \cdot\vec{e}_3 +g|\right|^2 \dd t\right)^{1/2}\\
& \le \dfrac{1}{2}\|f-g\|_{L^2(0,T)} +\dfrac{1}{2}\left(\int_{0}^{T}\left|(\overrightarrow{O P_i} \cdot\vec{e}_3 +f)-(\overrightarrow{O P_i} \cdot\vec{e}_3 +g)\right|^2 \dd t\right)^{1/2}\\
&= \dfrac{1}{2}\|f-g\|_{L^2(0,T)}+ \dfrac{1}{2}\|f-g\|_{L^2(0,T)}=\|f-g\|_{L^2(0,T)},
\end{align*}
where the first inequality is obtained as a result of an application of the Minkowski inequality, and the second last inequality is a direct consequence of the triangle inequality for the absolute value.

The Lipschitz continuity is thus established and we note in passing that the Lipschitz constant is equal to one, as it was to be proved.
\end{proof}

Thanks to the properties established in Theorem~\ref{betai}, we are in a position to define the nonlinear operator $\bm{N}:L^2(0,T;\mathbb{R}^{33}) \to L^2(0,T;\mathbb{R}^{33})$ by
$$
\bm{N}(\bm{V}):=\left((\gamma^\ast \beta_i \gamma)(\vec{v}_i)\right)_{i=1}^{11}, \quad\textup{ for all }\bm{V}=(\vec{v}_i)_{i=1}^{11} \in L^2(0,T;\mathbb{R}^{33}).
$$

It comes natural to define the following duality product
$$
\langle\langle \bm{N}(\bm{V}), \bm{W} \rangle\rangle_{L^2(0,T);\mathbb{R}^{33},L^2(0,T);\mathbb{R}^{33}} := \sum_{i=1}^{11} \int_{0}^{T} (\beta_i \gamma)(\vec{v}_i) \gamma(\vec{w}_i)\dd t,
$$
for all $\bm{V}=(\vec{v}_i)_{i=1}^{11}$ and all $\bm{W}=(\vec{w}_i)_{i=1}^{11}$ in $L^2(0,T;\mathbb{R}^{33})$.

The following properties easily descend from Theorem~\ref{betai}.

\begin{corollary}
\label{N}
The nonlinear operator $\bm{N}$ is monotone, bounded and Lipschitz continuous from $L^2(0,T;\mathbb{R}^{33})$ into itself.
\end{corollary}
\begin{proof}
To verify the monotonicity, it suffices to observe that the linearity of the adjoint $\gamma^\ast$ gives:
\begin{align*}
&\langle\langle \bm{N}(\bm{V}) - \bm{N}(\bm{W}), \bm{V} - \bm{W} \rangle\rangle_{L^2(0,T);\mathbb{R}^{33},L^2(0,T);\mathbb{R}^{33}}\\
&= \sum_{i=1}^{11} \int_{0}^{T} \left((\beta_i \gamma)(\vec{v}_i) - (\beta_i \gamma)(\vec{w}_i)\right) (\gamma(\vec{v}_i) - \gamma(\vec{w}_i)) \dd t\\
&\ge \sum_{i=1}^{11} \|\gamma(\vec{v}_i) - \gamma(\vec{w}_i)\|_{L^2(0,T)}^2 \ge 0,
\end{align*}
for all $\bm{V}=(\vec{v}_i)_{i=1}^{11}$ and all $\bm{W}=(\vec{w}_i)_{i=1}^{11}$ in $L^2(0,T;\mathbb{R}^{33})$.

The boundedness is immediate to verify since each operator $\gamma^\ast \beta_i \gamma: L^2(0,T;\mathbb{R}^{33}) \to L^2(0,T;\mathbb{R}^{33})$ is a composition of bounded operators.

For verifying the Lipschitz continuity, let us evaluate
\begin{align*}
&\|\bm{N}(\bm{V}) - \bm{N}(\bm{W})\|_{L^2(0,T;\mathbb{R}^{33})}=\left\|\sum_{i=1}^{11}\left|\gamma^\ast\left((\beta_i \gamma)(\vec{v}_i)-(\beta_i \gamma)(\vec{w}_i)\right)\right|\right\|_{L^2(0,T)}\\
&\le C_{\gamma^\ast} \sum_{i=1}^{11}\|\beta_i(\gamma(\vec{v}_i))-\beta_i(\gamma(\vec{w}_i))\|_{L^2(0,T)} \le  C_{\gamma^\ast} \sum_{i=1}^{11}\|\gamma(\vec{v}_i)-\gamma(\vec{w}_i)\|_{L^2(0,T)}\\
& \le C_{\gamma}C_{\gamma^\ast} \|\bm{V}-\bm{W}\|_{L^2(0,T;\mathbb{R}^{33})},
\end{align*}
for all $\bm{V}=(\vec{v}_i)_{i=1}^{11}$ and all $\bm{W}=(\vec{w}_i)_{i=1}^{11}$ in $L^2(0,T;\mathbb{R}^{33})$.
\end{proof}

We now recall a very important inequality which is used to study evolutionary problems: Gronwall's inequality (see the seminal paper~\cite{GW} or, for instance, Theorem~1.1 in Chapter~III of~\cite{Ha}).

\begin{theorem}
	\label{GW}
	Let $T>0$ and suppose that the function $y:[0,T]\to \mathbb R$ is absolutely continuous and such that
	$$
	\dfrac{\dd y}{\dd t}(t) \le a(t) y(t)+b(t), \textup{ a.e. in }(0,T),
	$$
	where $a, b \in L^1(0,T)$ and $a(t), b(t) \ge 0$ for a.a. $t\in (0,T)$.
	Then, it results
	$$
	y(t) \le \left[y(0)+\int_0^t b(s) \dd s\right]e^{\int_0^t a(s) \dd s}, \textup{ for all } t\in [0,T].
	$$
	\qed
\end{theorem}

Let us recall a compactness result proved by Simons (see, e.g., Corollary~4 of~\cite{Simons1987}), which will be used in what follows to recover the initial conditions. In what follows, the symbol "$\hookrightarrow$" denotes a \emph{continuous embedding}, whereas the symbol "$\hookrightarrow\hookrightarrow$" denotes a \emph{compact embedding}.

\begin{theorem}
	\label{Aubin}
	Let $T>0$ and let $X$, $Y$ and $Z$ be three Banach spaces such that
	$$
	X \hookrightarrow\hookrightarrow Y \hookrightarrow Z.
	$$
	
	Let $(f_n)_{n=1}^\infty$ be a bounded sequence in $L^\infty(0,T;X)$ and assume that the sequence of the weak derivatives in time $(\frac{\dd f_n}{\dd t})_{n=1}^\infty$ is bounded in $L^\infty(0,T;Z)$.
	Then, there exists a subsequence, still denoted $(f_n)_{n=1}^\infty$, that converges in the space $\mathcal{C}^0([0,T];Y)$.
	\qed
\end{theorem}

Let us also recall a result on vector-valued measures proved by Zinger in the paper~\cite{Singer1959} (see also, e.g., page~182 of~\cite{DieUhl1977}, and page~380 of~\cite{Dinculeanu1965}).

\begin{theorem}
	\label{ds}
	Let $\omega$ be a compact Hausdorff space and let $X$ be a Banach space satisfying the Radon-Nikodym property. Let $\mathcal{F}$ be the collection of Borel sets of $\omega$.
	
	There exists an isomorphism between $(\mathcal{C}^0(\omega;X))^\ast$ and the space of the regular Borel measures with finite variation taking values in $X^\ast$. In particular, for each $F \in (\mathcal{C}^0(\omega;X))^\ast$, there exists a unique regular Borel measure $\mu:\mathcal{F} \to X^\ast$ in $\mathcal{M}(\omega;X^\ast)$ with finite variation such that
	$$
	\langle\langle \alpha,F\rangle\rangle_{X} =\int_\omega {\phantom{,}}_{X^\ast}\langle \dd \mu,\alpha\rangle_{X},
	$$
	for all $\alpha \in \mathcal{C}^0(\omega;X)$.
	\qed
\end{theorem}

We are now ready to formulate the penalized variational formulation of the model under consideration. In what follows, the number $\kappa>0$ denotes the \emph{penalty parameter}. For sake of brevity, define the matrix
$$
\Upsilon:=\dfrac{k_s}{2}\Sigma^T\Sigma +\dfrac{k_b C^2}{2}\Theta^T\Theta,
$$
where, the matrices $\Sigma$ and $\Theta$, and the positive constants $k_s$, $k_b$ and $C$ are those defined in section~\ref{Sec:2}.

\begin{customprob}{$\mathcal{P}_\kappa$}
\label{penalty}
Given $\bm{F}=(\vec{f}_i)_{i=1}^{11} \in L^2(0,T;\mathbb{R}^{33})$, find $\bm{U}_\kappa=(\vec{u}_{i,\kappa})_{i=1}^{11}:[0,T] \to \mathbb{R}^{33}$ such that
\begin{align*}
\bm{U}_\kappa &\in L^\infty(0,T;\mathbb{R}^{33}),\\
\dfrac{\dd \bm{U}_\kappa}{\dd t} &\in L^\infty(0,T;\mathbb{R}^{33}),\\
\dfrac{\dd^2 \bm{U}_\kappa}{\dd t^2} &\in L^\infty(0,T;\mathbb{R}^{33}),
\end{align*}
that satisfies the following equations
$$
\dfrac{\dd^2 \bm{U}_\kappa}{\dd t^2} +\Upsilon \bm{U}_\kappa +\dfrac{1}{\kappa} \bm{N}(\bm{U}_\kappa) =\bm{F},
$$
in the sense of distributions in $(0,T)$, and satisfying the following initial conditions
\begin{align*}
\bm{U}_\kappa(0)&=\bm{U}_0 \in K,\\
\dfrac{\dd \bm{U}_\kappa}{\dd t}(0)&=\bm{U}_1 \in \mathbb{R}^{33},
\end{align*}
for a prescribed element $\bm{U}_0 \in K$ such that 
\begin{equation*}
(\overrightarrow{O P_i} + \vec{u}_{i,0}) \cdot \vec{e}_3 >0,\quad\textup{ for all }1 \le i \le 11,
\end{equation*}
and a prescribed element $\bm{U}_1 \in \mathbb{R}^{33}$.
\bqed
\end{customprob}

We first establish a series of preliminary results and we will then let the penalty parameter $\kappa$ approach zero. We will show that, by so doing, it is possible to recover a \emph{limit model}, which is the proposed model governing the deformation of a viral capsid subjected not to cross a given rigid surface in the time-dependent case.

\begin{lemma}
\label{lemma1}
For each $\kappa>0$, Problem~\ref{penalty} admits a unique solution $\bm{U}_\kappa \in W^{2,1}(0,T;\mathbb{R}^{33})$.
\end{lemma}
\begin{proof}
Let us put
\begin{align*}
\bm{X}_\kappa^{1}&:=\bm{U}_\kappa,\\
\bm{X}_\kappa^{2}&:=\dfrac{\dd \bm{U}_\kappa}{\dd t}=\dfrac{\dd \bm{X}_\kappa^{1}}{\dd t},
\end{align*}
so that the initial value problem in Problem~\eqref{penalty} can be written, for a.a. $t \in (0,T)$, in the form of a system of ordinary differential equations of the first order:
\begin{equation}
\label{system}
\begin{cases}
\dfrac{\dd \bm{X}_\kappa^{1}}{\dd t}&= \bm{X}_\kappa^{2},\\
\\
\dfrac{\dd \bm{X}_\kappa^{2}}{\dd t}&=\bm{F}-\Upsilon \bm{X}_\kappa^{1} -\dfrac{1}{\kappa} \bm{N}(\bm{X}_\kappa^{1}),\\
\\
\bm{X}_\kappa^{1}(0)&=\bm{U}_0,\\
\\
\dfrac{\dd \bm{X}_\kappa^{1}}{\dd t}(0)&=\bm{U}_1.
\end{cases}
\end{equation}

The mapping
\begin{equation}
\label{map}
\begin{aligned}
&(t,\bm{X}^{(1)}=(\vec{x}_i^{(1)})_{i=1}^{11},\bm{X}^{(2)}=(\vec{x}_i^{(2)})_{i=1}^{11}) \in (0,T) \times \mathbb{R}^{33} \times \mathbb{R}^{33} \\
&\quad\mapsto 
\begin{pmatrix}
\bm{X}^{(2)}\\
\bm{F}(t)-\Upsilon \bm{X}^{(1)} -\dfrac{1}{\kappa} \left(-\{(\overrightarrow{OP_i}+\vec{x}_i^{(1)})\cdot\vec{e}_3\}^{-}\right)_{i=1}^{11}
\end{pmatrix}
\end{aligned}
\end{equation}
is measurable with respect to $t$ and, thanks to Corollary~\ref{N}, is continuous with respect to $\bm{X}^{(1)}$ and $\bm{X}^{(2)}$. An application of the local-in-time existence and uniqueness theorem for systems of ordinary differential equations whose datum is a Carath\'eodory function (cf., e.g., Theorem~1.44 of~ \cite{Roubicek2005}) ensures the existence of a number $0 < \tau \le T$ such that the system~\eqref{system} admits a unique solution $(\bm{X}_\kappa^{1},\bm{X}_\kappa^{2}) \in W^{1,1}(0,\tau;\mathbb{R}^{33}) \times W^{1,1}(0,\tau;\mathbb{R}^{33})$.

We now compute the maximal interval for the locally unique solution. The idea consists in applying the weak version of the Cauchy-Lipschitz theorem for systems of first order ordinary differential equations where the datum appears in the form of a Carath\'eodory function (cf., e.g., Theorem~1.45 of~\cite{Roubicek2005}).

To this aim, observe that for each $\kappa>0$ the mapping~\eqref{map} is such that the norm of its right hand side is dominated by
$$
|\bm{X}^{(2)}| +\|\bm{F}\|_{L^2(0,T;\mathbb{R}^{33})} +\lambda_{\textup{max}} |\bm{X}^{(1)}| +\dfrac{1}{\kappa} |\bm{X}^{(1)}|+\dfrac{\max_{1 \le i \le 11} \{\overrightarrow{OP_i}\cdot\vec{e}_3\}}{\kappa},
$$
where $\lambda_{\textup{max}}>0$ is the maximal eigenvalue of the symmetric positive-definite matrix $\Upsilon$.

The aforementioned domination allows us to apply the weak version of the Cauchy-Lipschitz theorem for ordinary differential equations recalled beforehand, so as to infer that the maximal interval is the whole interval $[0,T]$ and, thus, that the unique solution of~\eqref{system} is defined over the whole interval $[0,T]$ up to a zero measure subset. This completes the proof.
\end{proof}

\begin{lemma}
\label{lemma2}
There exists a constant $c_0>0$ independent of $\kappa$ for which
\begin{align*}
\|\bm{U}_\kappa\|_{L^\infty(0,T;\mathbb{R}^{33})} &\le c_0, \quad \textup{ for all }\kappa>0,\\
\left\|\dfrac{\dd \bm{U}_\kappa}{\dd t}\right\|_{L^\infty(0,T;\mathbb{R}^{33})} &\le c_0, \quad \textup{ for all }\kappa>0,\\
\dfrac{1}{\kappa}\|\bm{N}(\bm{U}_\kappa)\|_{L^1(0,T;\mathbb{R}^{33})} &\le c_0, \quad \textup{ for all }\kappa>0,\\
\left\|\dfrac{\dd^2 \bm{U}_\kappa}{\dd t^2}\right\|_{L^1(0,T;\mathbb{R}^{33})} &\le c_0, \quad \textup{ for all }\kappa>0.
\end{align*}
\end{lemma}
\begin{proof}
Fix $t \in (0,T)$, and scalarly multiply the equations in Problem~\ref{penalty} by $\bm{V}=\frac{\dd \bm{U}_\kappa}{\dd t}(\tau)$, where $0<\tau<t$. We obtain that
\begin{equation}
	\label{int}
\begin{aligned}
&\dfrac{1}{2}\dfrac{\dd}{\dd t}\left|\dfrac{\dd \bm{U}_\kappa}{\dd t}(\tau)\right|^2+\dfrac{1}{2} \dfrac{\dd}{\dd t}\left(\bm{U}_\kappa(\tau)^T \Upsilon \bm{U}_\kappa(\tau)\right)\\
&\quad+\dfrac{1}{\kappa} \bm{N}(\bm{U}_\kappa)(\tau) \cdot \dfrac{\dd \bm{U}_\kappa}{\dd t}(\tau)=\bm{F}(\tau) \cdot \dfrac{\dd \bm{U}_\kappa}{\dd t}(\tau).
\end{aligned}
\end{equation}

Let us recall that, by Stampacchia's theorem (cf., e.g., \cite{Stampacchia1965}), we have that if $f \in W^{1,p}(0,T)$ then its negative part $f^{-} \in W^{1,p}(0,T)$ and
$$
\dfrac{\dd}{\dd t} f^{-}=
\begin{cases}
-\dfrac{\dd f}{\dd t}&, \textup{ if } f<0,\\
0&, \textup{ if }f \ge 0.
\end{cases}
$$

Thanks to this result, we obtain that
\begin{equation*}
\begin{aligned}
&\sum_{i=1}^{11} \int_{0}^{t}  (\beta_i\gamma)(\vec{u}_{i,\kappa})(\tau)\dfrac{\dd}{\dd \tau}(\gamma(\vec{u})_{i,\kappa})(\tau)\dd \tau\\
&=\dfrac{1}{2}\sum_{i=1}^{11} \int_{0}^{t} \dfrac{\dd}{\dd \tau}\left(|(\beta_i \gamma)(\vec{u}_{i,\kappa})(\tau)|^2\right)\dd \tau\\
&=\dfrac{1}{2}\sum_{i=1}^{11}|(\beta_i \gamma)(\vec{u}_{i,\kappa})(t)|^2 \ge 0,
\end{aligned}
\end{equation*}
where $(\beta_i \gamma)(\vec{u}_{i,0})=0$ for all $1 \le i \le 11$, since $\bm{U}_0=(\vec{u}_{0,i})_{i=1}^{11} \in K$.

Keeping the latter in mind, letting $\lambda_{\textup{min}}>0$ denote the minimal eigenvalue of the symmetric positive-definite matrix $(\Sigma^T \Sigma + \Theta^T \Theta)$, and applying Young's inequality~\cite{Young1912} gives:
\begin{equation}
	\label{est}
\begin{aligned}
&\dfrac{1}{2}\dfrac{\dd}{\dd t}\left\{\left|\dfrac{\dd \bm{U}_\kappa}{\dd t}(\tau)\right|^2+\left(\bm{U}_\kappa(\tau)^T \Upsilon \bm{U}_\kappa(\tau)\right)\right\}\\
&\le\dfrac{1}{2}\dfrac{\dd}{\dd t}\left\{\left|\dfrac{\dd \bm{U}_\kappa}{\dd t}(\tau)\right|^2+\left(\bm{U}_\kappa(\tau)^T \Upsilon \bm{U}_\kappa(\tau)\right)\right\}\\
&\quad+\dfrac{1}{2\kappa}\sum_{i=1}^{11}|(\beta_i \gamma)(\vec{u}_{i,\kappa})(\tau)|^2
\le \dfrac{|\bm{F}(\tau)|^2}{2}+\dfrac{1}{2}\left|\dfrac{\dd \bm{U}_\kappa}{\dd \tau}(\tau)\right|^2\\
&\le \dfrac{|\bm{F}(\tau)|^2}{2}+\dfrac{1}{2}\left|\dfrac{\dd \bm{U}_\kappa}{\dd \tau}(\tau)\right|^2+\dfrac{1}{2}\left(\bm{U}_\kappa(\tau)^T \Upsilon \bm{U}_\kappa(\tau)\right).
\end{aligned}
\end{equation}

Letting
\begin{align*}
y(\tau)&:=\dfrac{\dd}{\dd t}\left\{\dfrac{1}{2}\left|\dfrac{\dd \bm{U}_\kappa}{\dd t}(\tau)\right|^2+\dfrac{1}{2}\left(\bm{U}_\kappa(\tau)^T \Upsilon \bm{U}_\kappa(\tau)\right)\right\},\\
a(\tau)&:=1,\\
b(\tau)&:=\dfrac{1}{2}|\bm{F}(\tau)|^2,
\end{align*}
an application of the Gronwall's inequality (Theorem~\ref{GW}) to~\eqref{est} gives:
\begin{equation}
\label{bound}
\begin{aligned}
&\{\bm{U}_\kappa\}_{\kappa>0} \textup{ is bounded in }L^\infty(0,T;\mathbb{R}^{33}),\\
&\left\{\dfrac{\dd \bm{U}_\kappa}{\dd t}\right\}_{\kappa>0} \textup{ is bounded in }L^\infty(0,T;\mathbb{R}^{33}).
\end{aligned}
\end{equation}

As a result of~\eqref{bound}, we have that
\begin{equation*}
\dfrac{1}{2\kappa}\sum_{i=1}^{11}|(\beta_i \gamma)(\vec{u}_{i,\kappa})(t)|^2 \le C, \quad\textup{ for a.a. }t\in(0,T),
\end{equation*}
for some $C=C(\bm{F},T)>0$ so as to infer that:
\begin{equation*}
\{(\overrightarrow{OP_i}+\vec{u}_{i,\kappa})\cdot\vec{e}_3\}^{-} \to 0 \textup{ as }\kappa \to 0 \textup{ for all }1\le i \le 11.
\end{equation*}

Let us integrate in $(t_1,t_2) \subset (0,T)$ the equation in Problem~\ref{penalty} and let us pass to the absolute value; we obtain that
\begin{equation}
\label{bound1}
\begin{aligned}
&\dfrac{1}{\kappa} \left|\int_{t_1}^{t_2}\bm{N}(\bm{U}_\kappa) \dd t\right| =\left|\int_{t_1}^{t_2}\bm{F} \dd t-\int_{t_1}^{t_2} \Upsilon \bm{U}_\kappa \dd t-\int_{t_1}^{t_2} \dfrac{\dd^2 \bm{U}_\kappa}{\dd t^2} \dd t\right|\\
&\le\int_{t_1}^{t_2}|\bm{F}| \dd t+\lambda_{\textup{max}}\int_{t_1}^{t_2} |\bm{U}_\kappa| \dd t+\left|\dfrac{\dd \bm{U}_\kappa}{\dd t}(t_2)-\dfrac{\dd \bm{U}_\kappa}{\dd t}(t_1)\right|,
\end{aligned}
\end{equation}
where the evaluation of the first derivatives at $t_2$ and $t_1$ in the last term makes sense since for continuous functions the essential supremum coincides with the supremum, and since we have shown in Lemma~\ref{lemma1} that $\bm{U}_\kappa \in W^{2,1}(0,T;\mathbb{R}^{33})$. The boundednesses established in~\eqref{bound} thus give that there exists a constant $c_0>0$ independent of $\kappa$ such that:
$$
\dfrac{1}{\kappa} \left|\int_{t_1}^{t_2}\bm{N}(\bm{U}_\kappa) \dd t\right| \le c_0, \quad\textup{ for all }\kappa >0.
$$

Let us now observe that we can write
\begin{align*}
&\dfrac{1}{\kappa} \left|\int_{t_1}^{t_2}\bm{N}(\bm{U}_\kappa) \dd t\right| = \dfrac{1}{\kappa}\sup_{\substack{\bm{V}=(\vec{v}_i)_{i=1}^{11}\\|\bm{V}|=1}} \left|\left(\int_{t_1}^{t_2}\bm{N}(\bm{U}_\kappa) \dd t\right) \cdot \bm{V}\right|\\
&=\dfrac{1}{\kappa}\sup_{\substack{\bm{V}=(\vec{v}_i)_{i=1}^{11}\\|\bm{V}|=1}} \left|\int_{t_1}^{t_2}(\bm{N}(\bm{U}_\kappa)(t) \cdot \bm{V}) \dd t\right|\\
&\dfrac{1}{\kappa}\sup_{\substack{\bm{V}=(\vec{v}_i)_{i=1}^{11}\\|\bm{V}|=1}}\left|\sum_{i=1}^{11}\int_{t_1}^{t_2} \left(-\{(\overrightarrow{OP_i}+\vec{u}_{i,\kappa}(t))\cdot\vec{e}_3\}^{-}\right) (\vec{v}_i \cdot \vec{e}_3) \dd t\right|.
\end{align*}

By the definition of inner product in $\mathbb{R}^3$, the quantity $\vec{v}_i \cdot \vec{e}_3=|\vec{v}_i| \cos\widehat{\vec{v}_i \vec{e}_3}$ is maximized for vectors $\vec{v}_i$ which are parallel to $\vec{e}_3$. If, moreover, the vector $\vec{v}_i$ points in the direction opposite to $\vec{e}_3$ then all the factors appearing in the last integral are nonnegative. Having found one element with these features, the monotonicity of the integral gives:
\begin{align*}
&\sup_{\substack{\bm{V}=(\vec{v}_i)_{i=1}^{11}\\|\bm{V}|=1}}\left|\dfrac{1}{\kappa}\sum_{i=1}^{11}\int_{t_1}^{t_2} \left(-\{(\overrightarrow{OP_i}+\vec{u}_{i,\kappa}(t))\cdot\vec{e}_3\}^{-}\right) (\vec{v}_i \cdot \vec{e}_3) \dd t\right|\\
&=\dfrac{1}{\kappa}\sup_{\substack{\bm{V}=(\vec{v}_i)_{i=1}^{11}\\|\bm{V}|=1}}\left(\sum_{i=1}^{11}\int_{t_1}^{t_2} \left(-\{(\overrightarrow{OP_i}+\vec{u}_{i,\kappa}(t))\cdot\vec{e}_3\}^{-}\right) (\vec{v}_i \cdot \vec{e}_3) \dd t\right)\\
&=\dfrac{1}{\kappa}\int_{t_1}^{t_2} \sup_{\substack{\bm{V}=(\vec{v}_i)_{i=1}^{11}\\|\bm{V}|=1}}\left(\sum_{i=1}^{11}\left(-\{(\overrightarrow{OP_i}+\vec{u}_{i,\kappa}(t))\cdot\vec{e}_3\}^{-}\right) (\vec{v}_i \cdot \vec{e}_3)\right) \dd t\\
&=\dfrac{1}{\kappa}\int_{t_1}^{t_2} \sup_{\substack{\bm{V}=(\vec{v}_i)_{i=1}^{11}\\|\bm{V}|=1}}\left|\sum_{i=1}^{11}\left(-\{(\overrightarrow{OP_i}+\vec{u}_{i,\kappa}(t))\cdot\vec{e}_3\}^{-}\right) (\vec{v}_i \cdot \vec{e}_3)\right| \dd t\\
&=\dfrac{1}{\kappa}\int_{t_1}^{t_2} \sup_{\substack{\bm{V}=(\vec{v}_i)_{i=1}^{11}\\|\bm{V}|=1}}|\bm{N}(\bm{U}_\kappa) \cdot \bm{V}|\dd t=\dfrac{1}{\kappa}\int_{t_1}^{t_2} |\bm{N}(\bm{U}_\kappa)|\dd t.
\end{align*}

In conclusion, taking $t_1=0$ and $t_2=T$, we have that there exists a constant $c_0>0$ independent of $\kappa$ for which:
\begin{equation}
\label{bound2}
\dfrac{1}{\kappa}\|\bm{N}(\bm{U}_\kappa)\|_{L^1(0,T;\mathbb{R}^{33})}=\dfrac{1}{\kappa}\int_{0}^{T} |\bm{N}(\bm{U}_\kappa)|\dd t = \dfrac{1}{\kappa} \left|\int_{0}^{T}\bm{N}(\bm{U}_\kappa) \dd t\right| \le c_0, \quad\textup{ for all }\kappa >0.
\end{equation}

An application of~\eqref{bound2} to the equations of Problem~\ref{penalty} gives
\begin{equation}
\label{bound3}
\left\{\dfrac{\dd^2 \bm{U}_\kappa}{\dd t^2}\right\}_{\kappa>0} \textup{ is bounded in }L^1(0,T;\mathbb{R}^{33}).
\end{equation}

This completes the proof.

\end{proof}

In view of Lemma~\ref{lemma1} and Lemma~\ref{lemma2}, we let the parameter $\kappa$ approach zero in the penalized variational equations in Problem~\ref{penalty}. In what follows, weak convergences and weak-star convergences are respectively denoted by $\rightharpoonup$ and $\wsc$.
We observe that the first two boundedness properties in~\eqref{bound} imply that $\bm{U}_\kappa \in \mathcal{C}^0([0,T];\mathbb{R}^{33})$.

Since the acceleration $\frac{\dd^2 \bm{U}_\kappa}{\dd t^2}$ is uniformly bounded in $L^1(0,T;\mathbb{R}^{33})$, by the Dinculeanu-Zinger theorem we are able to extract a subsequence, still denoted $\{\bm{U}_\kappa\}_{\kappa>0}$, and to find a function $\bm{U} \in L^\infty(0,T;\mathbb{R}^{33})$ and a vector-valued measure $\bm{\mu}\in \mathcal{M}([0,T];\mathbb{R}^{33})$ such that:
\begin{equation}
\label{conv}
\begin{aligned}
\bm{U}_\kappa &\wsc \bm{U} \textup{ in } L^\infty(0,T;\mathbb{R}^{33}),\\
\dfrac{\dd \bm{U}_\kappa}{\dd t} &\wsc \dfrac{\dd \bm{U}}{\dd t} \textup{ in } L^\infty(0,T;\mathbb{R}^{33}),\\
\dfrac{\dd^2 \bm{U}_\kappa}{\dd t^2} &\wsc \bm{\mu} \textup{ in } \mathcal{M}([0,T];\mathbb{R}^{33}) \cong (\mathcal{C}^0([0,T];\mathbb{R}^{33}))^\ast.
\end{aligned}
\end{equation}

Besides, by the Aubin-Lions-Simon lemma (Theorem~\ref{Aubin}), we have that
\begin{equation}
\label{ALS}
\bm{U}_\kappa \to \bm{U},\quad\textup{ in }\mathcal{C}^0([0,T];\mathbb{R}^{33}).
\end{equation}

This leads us to define the linear and continuous operator $L_0:\mathcal{C}^0([0,T];\mathbb{R}^{33}) \to \mathbb{R}^{33}$ by
$$
L_0(\bm{V}):=\bm{V}(0),\quad\textup{ for all } \bm{V} \in \mathcal{C}^0([0,T];\mathbb{R}^{33}).
$$

Thanks to the convergence process~\eqref{conv}, we infer that
\begin{equation*}
\begin{aligned}
\bm{U}_\kappa &\rightharpoonup \bm{U} \textup{ in }W^{1,\infty}(0,T;\mathbb{R}^{33}) \hookrightarrow \mathcal{C}^0([0,T];\mathbb{R}^{33}),
\end{aligned}
\end{equation*}
and so we have that (recall that in finite-dimensional spaces weak convergence and strong convergence coincide; cf., e.g., \cite{Brez11}):
\begin{align*}
\bm{U}_\kappa(0) &\to \bm{U}(0).
\end{align*}

Since $\bm{U}_\kappa(0)=\bm{U}_0\in K$ for all $\kappa>0$, we infer that $\bm{U}(0)=\bm{U}_0$ as well. Moreover, by the third boundedness in the statement of Lemma~\ref{lemma2} and the continuity of the operator $\bm{N}$, we infer that
$$
\bm{N}(\bm{U})=\bm{0} \textup{ in }L^2(0,T;\mathbb{R}^{33}),
$$
so that $\bm{U} \in \mathcal{K}$ and $\bm{U}(0)=\bm{U}_0$. Therefore, there exists a number $t_0=t_0(\bm{F})>0$ independent of $\kappa$ such that:
\begin{equation}
\label{t0}
(\overrightarrow{OP_i}+\vec{u}_{i}(t) ) \cdot \vec{e}_3 >0,\quad\textup{ for all } t \in [0,t_0] \textup{ and for all }1 \le i \le 11.
\end{equation}

The convergence~\eqref{ALS} means that for all $\varepsilon>0$ there exists a number $\kappa_\varepsilon >0$ such that for each $0<\kappa<\kappa_\varepsilon$ it results:
\begin{equation*}
\sup_{t \in [0,T]} |\bm{U}(t)-\bm{U}_\kappa(t)| <\varepsilon.
\end{equation*}

Hence, if $0<\kappa<\kappa_\varepsilon$, in correspondence of the number $t_0=t_0(\bm{F})>0$ defined beforehand we have that for all $t \in [0,t_0]$ and all $1 \le i \le 11$, it results:
\begin{equation}
\label{continuity}
\begin{aligned}
&(\overrightarrow{OP_i}+\vec{u}_{i,\kappa}(t) ) \cdot \vec{e}_3\\
&=(\overrightarrow{OP_i}+\vec{u}_{i,\kappa}(t)-\vec{u}_i(t)) \cdot \vec{e}_3 +\vec{u}_i(t) \cdot \vec{e}_3\\
&=(\overrightarrow{OP_i}+\vec{u}_i(t)) \cdot \vec{e}_3+(\vec{u}_{i,\kappa}(t)-\vec{u}_i(t))\cdot \vec{e}_3\\
&\ge (\overrightarrow{OP_i}+\vec{u}_i(t)) \cdot \vec{e}_3 - \sup_{t \in [0,T]}|\vec{u}_{i,\kappa}(t)-\vec{u}_i(t)|\\
&>(\overrightarrow{OP_i}+\vec{u}_i(t)) \cdot \vec{e}_3 -\varepsilon.
\end{aligned}
\end{equation}

In view of~\eqref{t0}, we choose 
$$
0<\varepsilon<\frac{1}{2}\inf_{\substack{t\in[0,t_0]\\1\le i \le 11}}(\overrightarrow{OP_i}+\vec{u}_i(t)) \cdot \vec{e}_3,
$$
so that the right-hand side in~\eqref{continuity} is strictly positive. We thus conclude that:
\begin{equation}
\label{stability}
(\overrightarrow{OP_i}+\vec{u}_{i,\kappa}(t) ) \cdot \vec{e}_3 >0,\quad\textup{ for all } t \in [0,t_0] \textup{ and for all }1 \le i \le 11,
\end{equation}
for all $0<\kappa<\kappa_\varepsilon$.

For what concerns the first derivative of $\bm{U}_\kappa$, the condition~\eqref{stability} implies that in the time interval $[0,t_0]$ the constraint is inactive and so the penalty term in Problem~\ref{penalty} vanishes in the time interval $[0,t_0]$. As a result, we have that
$$
\dfrac{\dd^2 \bm{U}_\kappa}{\dd t^2} \textup{ is bounded in } L^\infty(0,t_0;\mathbb{R}^{33}),
$$
which means, by the Sobolev embedding theorem (cf., e.g., Theorem~10.1.20 of~\cite{YP}), that up to passing to subsequences:
\begin{equation}
\label{local}
\dfrac{\dd \bm{U}_\kappa}{\dd t} \rightharpoonup \dfrac{\dd \bm{U}}{\dd t} \textup{ in }\mathcal{C}^0([0,t_0];\mathbb{R}^{33}).
\end{equation}

This leads us to define the linear and continuous operator $L_1:\mathcal{C}^0([0,t_0];\mathbb{R}^{33}) \to \mathbb{R}^{33}$ by
$$
L_1(\bm{V}):=\bm{V}(0),\quad\textup{ for all } \bm{V} \in \mathcal{C}^0([0,t_0];\mathbb{R}^{33}).
$$

By virtue of~\eqref{local}, we are in a position to infer that:
\begin{equation*}
\dfrac{\dd \bm{U}}{\dd t}(0)=\bm{U}_1.
\end{equation*}

This shows that the limit vector field $\bm{U} \in L^\infty(0,T;\mathbb{R}^{33})$ satisfies $\frac{\dd\bm{U}}{\dd t} \in L^\infty(0,T;\mathbb{R}^{33})$ as well as the initial conditions.

Let us now multiply, in the sense of the duality $\langle\langle \cdot, \cdot \rangle\rangle_{\mathcal{M}([0,T];\mathbb{R}^{33}),\mathcal{C}^0([0,T];\mathbb{R}^{33})}$ of $\mathcal{C}^0([0,T];\mathbb{R}^{33})$ and its dual (which is identified, by Theorem~\ref{ds}, with the space $\mathcal{M}([0,T];\mathbb{R}^{33})$), the equations in Problem~\ref{penalty} by $(\bm{V}-\bm{U}_\kappa)$, where $\bm{V} \in \mathcal{K}$. We obtain that:
\begin{equation}
\label{step1}
\begin{aligned}
&\int_{0}^{T} \dfrac{\dd^2 \bm{U}_\kappa}{\dd t^2} \cdot (\bm{V}-\bm{U}_\kappa) \dd t\\
&+\int_{0}^{T}(\bm{V}-\bm{U}_\kappa)^T \Upsilon \bm{U}_\kappa\dd t\\
&+\dfrac{1}{\kappa}\int_{0}^{T} \bm{N}(\bm{U}_\kappa) \cdot (\bm{V}-\bm{U}_\kappa) \dd t\\
&=\int_{0}^{T} \bm{F} \cdot (\bm{V}-\bm{U}_\kappa) \dd t,\quad\textup{ for all }\bm{V} \in \mathcal{K}.
\end{aligned}
\end{equation}

Observe that the monotonicity of $\bm{N}$ established in Corollary~\ref{N} gives:
$$
\dfrac{1}{\kappa}\int_{0}^{T} \bm{N}(\bm{U}_\kappa) \cdot (\bm{V}-\bm{U}_\kappa) \dd t=-\dfrac{1}{\kappa} \int_{0}^{T} (\bm{N}(\bm{V})-\bm{N}(\bm{U}_\kappa) )\cdot (\bm{V}-\bm{U}_\kappa) \dd t \le 0.
$$

Therefore, the equations~\eqref{step1} become:
\begin{equation}
	\label{step2}
	\begin{aligned}
		&\int_{0}^{T} \dfrac{\dd^2 \bm{U}_\kappa}{\dd t^2} \cdot (\bm{V}-\bm{U}_\kappa) \dd t\\
		&+\int_{0}^{T}(\bm{V}-\bm{U}_\kappa)^T \Upsilon \bm{U}_\kappa\dd t\\
		&\ge\int_{0}^{T} \bm{F} \cdot (\bm{V}-\bm{U}_\kappa) \dd t,\quad\textup{ for all }\bm{V} \in \mathcal{K}.
	\end{aligned}
\end{equation}

Exploiting the convergences in~\eqref{conv} and~\eqref{ALS} allows us to change~\eqref{step2} into:
\begin{equation}
	\label{step3}
	\begin{aligned}
		&\langle\langle\bm{\mu},\bm{V}-\bm{U}\rangle\rangle_{\mathcal{M}([0,T];\mathbb{R}^{33}),\mathcal{C}^0([0,T];\mathbb{R}^{33})}\\
		&+\int_{0}^{T}(\bm{V}-\bm{U})^T \Upsilon \bm{U}\dd t\\
		&\ge\int_{0}^{T} \bm{F} \cdot (\bm{V}-\bm{U}) \dd t,\quad\textup{ for all }\bm{V} \in \mathcal{K}.
	\end{aligned}
\end{equation}

We can observe that, by the convergence process~\eqref{conv}, the vector-valued measure $\bm{\mu} \in \mathcal M([0,T];\mathbb{R}^{33})$ can be interpreted as the acceleration of the limit displacement $\bm{U}$. Indeed, by the classical definition of weak derivative, we have that
$$
\int_{0}^{T} \dfrac{\dd \bm{U}_\kappa}{\dd t}  \varphi' \dd t = -\int_{0}^{T} \dfrac{\dd^2 \bm{U}_\kappa}{\dd t^2} \varphi \dd t, \quad\mbox{ for all }\varphi \in \mathcal{D}(0,T).
$$

By the properties of Lebesgue-Bochner integrals we have that, for all $\bm{V}\in \mathbb{R}^{33}$ and all $\varphi \in \mathcal{D}(0,T)$, it results
$$
\int_{0}^{T}\dfrac{\dd \bm{U}_{\kappa}}{\dd t} \cdot (\varphi' \bm{V}) \dd t=-\int_{0}^T \dfrac{\dd^2 \bm{U}_\kappa}{\dd t^2} \cdot (\varphi \bm{V}) \dd t,
$$
so that, letting $\kappa \to 0$ (see Comment~3 of Chapter~4 of~\cite{Brez11}) gives
\begin{align*}
	&\left(\int_{0}^T \dfrac{\dd \bm{U}_\kappa}{\dd t} \varphi' \dd t\right) \cdot \bm{V}
	= \int_{0}^T\left(\dfrac{\dd \bm{U}_\kappa}{\dd t} \cdot \bm{V}\right) \varphi' \dd t\\
	=&-\langle\langle\bm{\mu}, \varphi \bm{V} \rangle\rangle_{\mathcal{M}([0,T];\mathbb{R}^{33}),\mathcal{C}^0([0,T];\mathbb{R}^{33})}
	=-\int_{0}^T\dd \bm{\mu} \cdot (\varphi \bm{V}) \dd t,	
\end{align*}
where the first equality holds by Fubini's theorem, the second equality holds by Theorem~\ref{ds}, the third convergence of the process~\eqref{conv} and the definition of weak derivative, and, finally, the last equality holds true by Theorem~\ref{ds}. This means that it is licit to replace the measure $\bm{\mu}$ in the third convergence of~\eqref{conv} by the more intuitive symbol $\frac{\dd^2 \bm{U}}{\dd t^2}$.

In conclusion, we have shown that there exists at least one element $\bm{U} \in \mathcal{K}$ satisfying the following initial value \emph{limit} problem.

\begin{customprob}{$\mathcal{P}$}
	\label{limit}
	Given $\bm{F}=(\vec{f}_i)_{i=1}^{11} \in L^2(0,T;\mathbb{R}^{33})$, find $\bm{U}=(\vec{u}_{i})_{i=1}^{11}:[0,T] \to \mathbb{R}^{33}$ such that
	\begin{align*}
		\bm{U} &\in \mathcal{K},\\
		\dfrac{\dd \bm{U}}{\dd t} &\in L^\infty(0,T;\mathbb{R}^{33}),\\
		\dfrac{\dd^2 \bm{U}}{\dd t^2} &\in \mathcal{M}([0,T];\mathbb{R}^{33}),
	\end{align*}
	that satisfies the following variational inequalities
	\begin{equation*}
	\begin{aligned}
	&\left\langle\left\langle\dfrac{\dd^2 \bm{U}}{\dd t^2},\bm{V}-\bm{U}\right\rangle\right\rangle_{\mathcal{M}([0,T];\mathbb{R}^{33}),\mathcal{C}^0([0,T];\mathbb{R}^{33})}\\
	&\quad+\int_{0}^{T} (\bm{V}-\bm{U})^T \Upsilon \bm{U} \dd t \ge \int_{0}^{T}\bm{F} \cdot (\bm{V}-\bm{U}) \dd t,
	\end{aligned}
	\end{equation*}
	for all $\bm{V}=(\vec{v}_i)_{i=1}^{11} \in \mathcal{K}$, and satisfying the following initial conditions
	\begin{align*}
		\bm{U}(0)&=\bm{U}_0,\\
		\dfrac{\dd \bm{U}}{\dd t}(0)&=\bm{U}_1,
	\end{align*}
	for some prescribed elements $\bm{U}_0 \in K$ and $\bm{U}_1 \in \mathbb{R}^{33}$ as in the statement of Problem~\ref{penalty}.
\bqed
\end{customprob}

We summarize the results that we proved in the following theorem, which constitutes the main result of this section.

\begin{theorem}
\label{th1}
Problem~\ref{limit} admits at least one solution.
\qed
\end{theorem}

\section*{Acknowledgments}

This work was partly supported by the Research Fund of Indiana University.

B.D. acknowledges support from the Army Research Office, under award W911NF2010072, and from the National Science Foundation, under award CBET 1740432.

\bibliographystyle{abbrvnat} 
\bibliography{references.bib}	

\begin{thebibliography}{21}
\providecommand{\natexlab}[1]{#1}
\providecommand{\url}[1]{\texttt{#1}}
\expandafter\ifx\csname urlstyle\endcsname\relax
  \providecommand{\doi}[1]{doi: #1}\else
  \providecommand{\doi}{doi: \begingroup \urlstyle{rm}\Url}\fi

\bibitem[Abate and Tovena(2012)]{AbateTovena2012}
M.~Abate and F.~Tovena.
\newblock \emph{Curves and surfaces}, volume~55 of \emph{Unitext}.
\newblock Springer, Milan, 2012.
\newblock Translated from the 2006 Italian original by Daniele A. Gewurz.

\bibitem[Bock et~al.(2016)Bock, Jaru\v{s}ek, and
  \v{S}ilhav\'{y}]{BockJarSil2016}
I.~Bock, J.~Jaru\v{s}ek, and M.~\v{S}ilhav\'{y}.
\newblock On the solutions of a dynamic contact problem for a thermoelastic von
  {K}\'{a}rm\'{a}n plate.
\newblock \emph{Nonlinear Anal. Real World Appl.}, 32:\penalty0 111--135, 2016.

\bibitem[Brezis(2011)]{Brez11}
H.~Brezis.
\newblock \emph{Functional {A}nalysis, {S}obolev {S}paces and {P}artial
  {D}ifferential {E}quations}.
\newblock Springer, New York, 2011.

\bibitem[Ciarlet(2013)]{PGCLNFAA}
P.~G. Ciarlet.
\newblock \emph{Linear and nonlinear functional analysis with applications}.
\newblock Society for Industrial and Applied Mathematics, Philadelphia, PA,
  2013.
\newblock ISBN 978-1-611972-58-0.

\bibitem[Diestel and Uhl(1977)]{DieUhl1977}
J.~Diestel and J.~J. Uhl.
\newblock \emph{Vector measures}.
\newblock American Mathematical Society, Providence, R.I., 1977.

\bibitem[Dinculeanu(1967)]{Dinculeanu1965}
N.~Dinculeanu.
\newblock \emph{Vector measures}.
\newblock Pergamon Press, Oxford-New York-Toronto, Ont.; VEB Deutscher Verlag
  der Wissenschaften, Berlin, 1967.

\bibitem[Ekeland and Temam(1999)]{EkelandTemam1999}
I.~Ekeland and R.~Temam.
\newblock \emph{Convex analysis and variational problems}, volume~28 of
  \emph{Classics in Applied Mathematics}.
\newblock Society for Industrial and Applied Mathematics (SIAM), Philadelphia,
  PA, english edition, 1999.
\newblock Translated from the French.

\bibitem[Gronwall(1919)]{GW}
T.~H. Gronwall.
\newblock Note on the derivatives with respect to a parameter of the solutions
  of a system of differential equations.
\newblock \emph{Ann. of Math.}, 20\penalty0 (4):\penalty0 292--296, 1919.

\bibitem[Hartman(1982)]{Ha}
P.~Hartman.
\newblock \emph{Ordinary Differential Equations}.
\newblock Society for Industrial and Applied Mathematics, Philadelphia,
  {S}econd edition, 1982.

\bibitem[Helfrich(1973)]{Helfrich1973}
W.~Helfrich.
\newblock Elastic properties of lipid bilayers: Theory and possible
  experiments.
\newblock \emph{Zeitschrift f{\"{u}}r Naturforschung C}, 28\penalty0
  (11-12):\penalty0 693--703, 1973.

\bibitem[Kyritsi-Yiallourou and Papageorgiou(2009)]{YP}
S.~Kyritsi-Yiallourou and N.~S. Papageorgiou.
\newblock \emph{Handbook of Applied Analysis}.
\newblock Springer, New York, 2009.

\bibitem[Piersanti et~al.(Accepted)Piersanti, White, Dragnea, and
  Temam]{PieWhiDraTem2021}
P.~Piersanti, K.~White, B.~Dragnea, and R.~Temam.
\newblock A simplified model of virus deformation in contact with a surface.
\newblock \emph{{A}ppl. {A}nal.}, Accepted.
\newblock URL \url{https://arxiv.org/abs/2112.15340}.

\bibitem[Prasolov and Tikhomirov(2001)]{Prasolov2001}
V.~V. Prasolov and V.~M. Tikhomirov.
\newblock \emph{Geometry}, volume 200 of \emph{Translations of Mathematical
  Monographs}.
\newblock American Mathematical Society, Providence, RI, 2001.
\newblock Translated from the 1997 Russian original by O. V. Sipacheva.

\bibitem[Roub\'{\i}\v{c}ek(2005)]{Roubicek2005}
T.~Roub\'{\i}\v{c}ek.
\newblock \emph{Nonlinear partial differential equations with applications},
  volume 153 of \emph{International Series of Numerical Mathematics}.
\newblock Birkh\"{a}user Verlag, Basel, 2005.
\newblock ISBN 978-3-7643-7293-4; 3-7643-7293-1.

\bibitem[Simon(1987)]{Simons1987}
J.~Simon.
\newblock Compact sets in the space {$L^p(0,T;B)$}.
\newblock \emph{Ann. Mat. Pura Appl. (4)}, 146:\penalty0 65--96, 1987.
\newblock ISSN 0003-4622.
\newblock \doi{10.1007/BF01762360}.
\newblock URL \url{https://doi.org/10.1007/BF01762360}.

\bibitem[Stampacchia(1966)]{Stampacchia1965}
G.~Stampacchia.
\newblock \emph{\`Equations elliptiques du second ordre \`a coefficients
  discontinus}, volume 1965 of \emph{S\'{e}minaire de Math\'{e}matiques
  Sup\'{e}rieures, No. 16 (\'{E}t\'{e}}.
\newblock Les Presses de l'Universit\'{e} de Montr\'{e}al, Montreal, Que.,
  1966.

\bibitem[Sun and Yuan(2006)]{SunYuan2006}
W.~Sun and Y.-X. Yuan.
\newblock \emph{Optimization theory and methods}, volume~1 of \emph{Springer
  Optimization and Its Applications}.
\newblock Springer, New York, 2006.
\newblock Nonlinear programming.

\bibitem[Young(1912)]{Young1912}
W.~H. Young.
\newblock On {C}lasses of {S}ummable {F}unctions and their {F}ourier {S}eries.
\newblock \emph{{P}roc. {R}. {S}oc. {L}ond. {S}er. {A} {M}ath. {P}hys. {E}ng.
  {S}ci.}, 87:\penalty0 225--229, 1912.

\bibitem[Zandi et~al.(2020)Zandi, Dragnea, Travesset, and Podgornik]{Zandi2020}
R.~Zandi, B.~Dragnea, A.~Travesset, and R.~Podgornik.
\newblock On virus growth and form.
\newblock \emph{Phys. Rep.}, 847:\penalty0 1--102, 2020.

\bibitem[Zeng et~al.(2017)Zeng, Hernando-P{\'{e}}rez, Dragnea, Ma, van~der
  Schoot, and Zandi]{Zeng2017a}
C.~Zeng, M.~Hernando-P{\'{e}}rez, B.~Dragnea, X.~Ma, P.~van~der Schoot, and
  R.~Zandi.
\newblock {Contact Mechanics of a Small Icosahedral Virus}.
\newblock \emph{Phys. Rev. Lett.}, 119\penalty0 (3):\penalty0 038102, jul 2017.
\newblock ISSN 0031-9007.
\newblock \doi{10.1103/PhysRevLett.119.038102}.
\newblock URL \url{http://link.aps.org/doi/10.1103/PhysRevLett.119.038102}.

\bibitem[Zinger(1957)]{Singer1959}
I.~Zinger.
\newblock Linear functionals on the space of continuous mappings of a compact
  {H}ausdorff space into a {B}anach space.
\newblock \emph{Rev. Math. Pures Appl.}, 2:\penalty0 301--315, 1957.

\end{thebibliography}

\end{document}